\documentclass[leqno,final]{article}

\usepackage[color,notref,notcite]{showkeys}
\definecolor{labelkey}{rgb}{0.6,0,1}

\usepackage{amsfonts,amsmath,amsthm}
\usepackage{showkeys}
\usepackage{amsmath}
\usepackage{epsfig}
\usepackage{graphicx}

\def\eps{\varepsilon}

\def\R{\hbox{\bf R}}
\def\Z{\hbox{\bf Z}}

\def\N{\hbox{\bf N}}

\def\D{{\Delta}}

\def\I{{\cal I}}
\def\F{{\cal F}}

\def\a{\alpha}

\def\e{\epsilon}

\def\l{{\lambda}}

\def\d{\delta}

\renewcommand{\thesubsection}{\arabic{section}.\arabic{subsection}}

\renewcommand{\theequation}{\arabic{section}.\arabic{equation}}

\newtheorem{theo}{\bf Theorem}[section]
\newtheorem{lem}[theo]{\bf Lemma}
\newtheorem{pro}[theo]{\bf Proposition}
\newtheorem{cor}[theo]{\bf Corollary}
\newtheorem{defi}[theo]{\bf Definition}
\newtheorem{ex}{Example}
\theoremstyle{remark}
\newtheorem{rem}[theo]{\bf Remark}

\renewcommand{\thesubsection}{\arabic{section}.\arabic{subsection}}

\renewcommand{\theequation}{\arabic{section}.\arabic{equation}}

\renewcommand{\N}{{\mathbb N}}
\renewcommand{\R}{{\mathbb R}}

\renewcommand{\Z}{{\mathbb Z}}
\renewcommand{\e}{\varepsilon}

\newcommand{\ox}{\overline{x}}

\newcommand{\ot}{\overline{t}}

\newenvironment{Proofc}[1]{\smallskip\par\noindent\textsc{#1}\quad}%
  {\hfill$\Box$\bigskip\par}

\begin{document}

\title{\bf Homogenization of accelerated Frenkel-Kontorova models
  with $n$ types of particles}
\author{\renewcommand{\thefootnote}{\arabic{footnote}}
  N. Forcadel\footnotemark[1] , C. Imbert\footnotemark[1] ,
  R. Monneau\footnotemark[2]}

\footnotetext[1]{CEREMADE, UMR CNRS 7534, Universit\'e Paris-Dauphine,
Place de Lattre de Tassigny, 75775 Paris Cedex 16, France}
\footnotetext[2]{Universite Paris-Est, Cermics, Ecole des ponts, 6-8 avenue Blaise Pascal, 77455 Marne la
Vallee Cedex  2, France.}

\maketitle

\vspace{20pt}


\begin{abstract}
  We consider systems of ODEs that describe the dynamics of
  particles. Each particle satisfies a Newton law (including a damping 
  term and an acceleration term) where the force is created by the interactions
  with other particles and with a periodic potential. The presence
  of a damping term allows the system to be monotone. Our study takes
  into account the fact that the particles can be different.

  After a proper hyperbolic rescaling, we show that solutions of
  these systems of ODEs converge to solutions of some macroscopic
  homogenized Hamilton-Jacobi equations.
\end{abstract}

\noindent {\small \bf AMS Classification:} {\small 35B27, 35F20, 45K05, 47G20, 49L25, 35B10.}
\smallskip

\noindent{\small{\bf{Keywords:}}} {\small particle system, periodic
  homogenization, Frenkel-Kontorova models, Hamilton-Jacobi equations,
  hull function}

\section{Introduction}

The goal of this paper is to obtain homogenization results for the
dynamics of accelerated Frenkel-Kontorova type systems with $n$ types
of particles. The Frenkel-Kontorova model is a simple physical model 
used in various fields: mechanics, biology, chemistry \textit{etc.}
The reader is referred to \cite{BK} for a general presentation of
models and mathematical problems.  
In this introduction, we start with the simplest
accelerated Frenkel-Kontorova model where there is only one type of
particle (see Eq.~\eqref{eq:1}). We then explain how to deal with $n$
types of particles (see Eq.~\eqref{eq:1n}). We finally present the
general case, namely systems of ODEs of the following form (for a
fixed $m\in \N$)
\begin{equation}\label{eq:fkg}
m_0 \frac{d^2 U_i}{d \tau^2} + \frac{d U_i}{d \tau} = F_i (\tau,U_{i-m}, \dots, U_{i+m}) 
\end{equation}
where $U_i(\tau)$ denotes the position of the particle $i\in \Z$ at
the time $\tau$. Here, $m_0$ is the mass of the particle and $F_i$ is
the force acting on the particle $i$, which will be made precise
later. 

Remark the presence of the damping term $ \frac{d U_i}{d
  \tau}$ on the left hand side of the equation. If the mass
$m_0$ is assumed to be small enough, then this system is monotone. We 
will make such an assumption and the monotonicity of the system 
is crucial in our analysis. 

We recall that the case of fully overdamped dynamics, \textit{i.e.} for
$m_0=0$, has already been treated in \cite{FIM2} (for only one type of
particles).
\bigskip

Several results are related to our analysis. For instance in
\cite{CLL}, homogenization results are obtained for monotone
systems of Hamilton-Jacobi equations. Notice that they obtain a system
at the limit while we will obtain a single equation. Techniques from
dynamical systems are also used to study systems of ODEs; see for
instance \cite{DLL,JS} and references therein.

\subsection{The classical overdamped Frenkel-Kontorova model}

The classical Frenkel-Kontorova model  describes a chain of
classical particles evolving in a one dimensional space, coupled with their
neighbours and subjected to a periodic potential.  If $\tau$ denotes time
and $U_i (\tau)$ denotes the position of the particle $i \in \Z$, one of
the simplest FK models is given by the following dynamics
\begin{equation}\label{eq:1}
m_0 \frac{d^2 U_i }{d\tau^2} +
 \frac{d U_i }{d\tau}= U_{i+1}-2U_i+U_{i-1}  
+ \sin \left(2\pi U_i\right) + L 
\end{equation}
where $m_0$ denotes the mass of the particle, $L$ is a constant driving
force which can make the whole ``train of particles'' move and the term
$\sin \left(2\pi U_i\right)$ describes the force created by a periodic
potential whose period is assumed to be $1$.  Notice that in the previous
equation, we set to one physical constants in front of the elastic and the
exterior forces (friction and periodic potential). The goal of our work is
to describe what is the macroscopic behaviour of the solution $U$ of
\eqref{eq:1} as the number of particles per length unit goes to
infinity. As mentioned above, the particular case where $m_0=0$ is referred
to as the fully overdamped one and has been studied in \cite{FIM2}.

We would like next to give the flavour of our main results.  In order 
to do so, let us assume that at initial time,
particles satisfy
\begin{eqnarray*}
U_i(0)&=&\eps^{-1}u_0(i\eps) \\
\frac{d U_i}{d \tau} (0)&=&0
\end{eqnarray*}
for some $\eps>0$ and some Lipschitz continuous function $u_0(x)$
which satisfies the following assumption
\begin{itemize}
\item[] {\bf Initial gradient bounded from above and below}
\begin{equation}\label{eq:(A0)}
0< 1/K_0 \le (u_0)_x \le K_0 \quad \mbox{on}\quad \R 
\end{equation}
for some fixed $K_0>0$.
\end{itemize}
Such an assumption can be interpreted by saying that at initial time,
the number of particles per length unit lies in $(K_0^{-1} \eps^{-1},
K_0 \eps^{-1})$.

It is then natural to ask what is the macroscopic behaviour of the
solution $U$ of \eqref{eq:1} as $\eps$ goes to zero, \textit{i.e.} as
the number of particles per length unit goes to infinity. To this end,
we define the following function which describes the rescaled
positions of the particles
\begin{equation}\label{eq:ue}
  \overline{u}^\eps(t,x)=\eps U_{\lfloor \eps^{-1} x \rfloor}(\eps^{-1}t)   
\end{equation}
where $\lfloor \cdot \rfloor$ denotes the floor integer part.  One of
our main results states that the limiting dynamics as $\eps$ goes to
$0$ of \eqref{eq:1} is determined by a first order Hamilton-Jacobi
equation of the form
\begin{equation}\label{eq:3}
  \left\{\begin{array}{ll}
      u^0_t = \overline{F}(u^0_x) &\quad \mbox{for}\quad (t,x)\in (0,+\infty)\times \R,\\
      u^0(0,x)=u_0(x)   &\quad \mbox{for}\quad x \in \R
    \end{array}\right.
\end{equation}
where $\overline{F}$ is a continuous function to be determined.  More
precisely, we have the following homogenization result
\begin{theo}[\bf Homogenization of the accelerated FK model] \label{th:0} 
  There exists a critical value $m_0^c$ such that for all $m_0 \in]0, m_0^c]$ and all
  $L \in \R$, there exists a continuous function $\overline{F}: \R \to
  \R$ such that, under assumption \eqref{eq:(A0)}, the function
  $\overline{u}^\eps$ converges locally uniformly towards the unique
  viscosity solution $u^0$ of \eqref{eq:3}.
\end{theo}
\begin{rem}
The critical mass $m_0^c$ is made precise in Assumption (A3) below.
\end{rem}

\subsection{Example of systems with $n$ types of particles}

We now present the case of systems with $n$ types of particles.  Let us
start with the typical problem we have in mind. Let  $n\in
\N\backslash \left\{0\right\}$ be some integer and let us consider a sequence of real
numbers $(\theta_i)_{i\in\Z}$ such that
$$
\theta_{i+n}=\theta_i >0 \quad \mbox{for all}\quad i\in\Z \, .
$$
It is then natural to consider the generalized FK model with $n$
different types of particles that stay ordered on the real line. Then,
instead of satisfying \eqref{eq:1}, we can assume that $U_i$ satisfies
for $\tau \in (0,+\infty)$ and $i\in \Z$   
\begin{equation}\label{eq:1n}
  m_0 \frac {d^2 U_i}{d^2 \tau}+\frac{d U_i }{d\tau}
  = \theta_{i+1}(U_{i+1}-U_i) -\theta_i(U_i-U_{i-1}) + \sin \left(2\pi U_i\right) + L
\end{equation}

Such a model is sketched on figure \ref{f2}. As we shall see it, we can prove the same kind of 
homogenization results as Theorem \ref{th:0}.\medskip 
\begin{figure}[ht]
  \centering\epsfig{figure=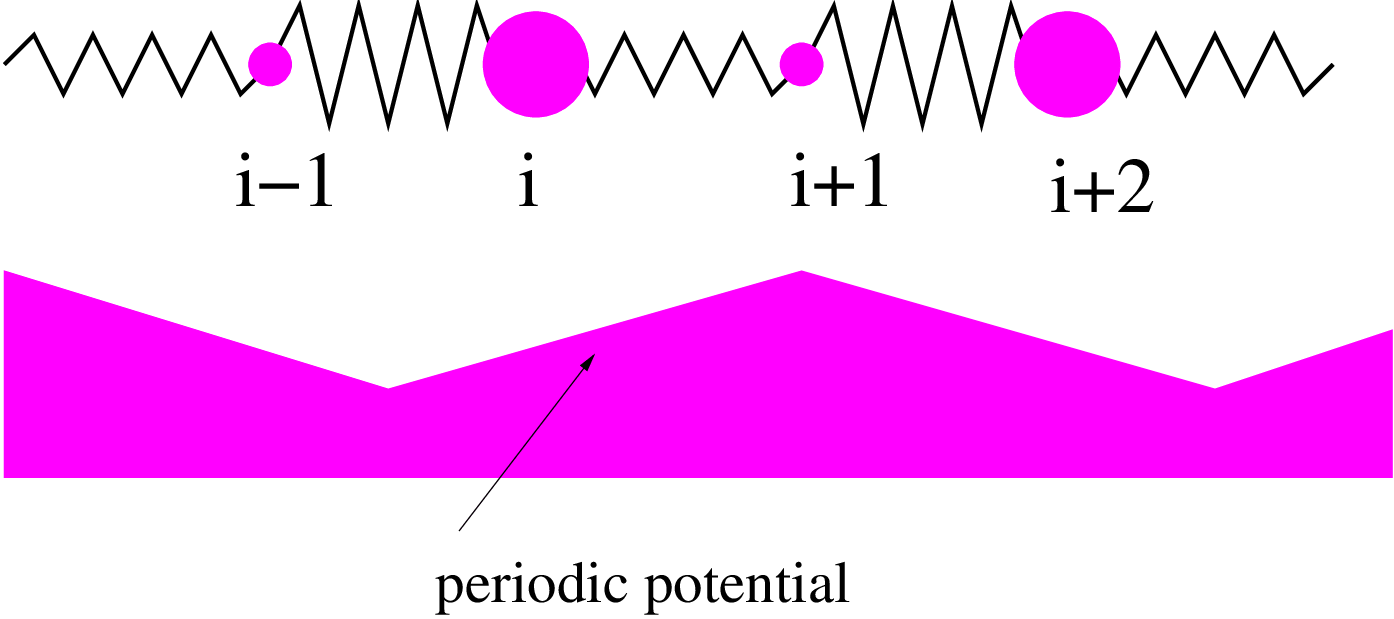,width=60mm}
  \caption{The FK model with $n=2$ type of particles (and of springs)
    and an interaction up to the $m=1$ neighbours}\label{f2}
\end{figure}

As we mentioned it before, it is crucial in our analysis to deal
with monotone systems of ODEs. Inspired of the work of Baesens and
MacKay \cite{BM04} and of Hu, Qin and Zheng \cite{HQZ}, we introduce
for all $i\in \Z$ the following function
$$
\Xi_i(\tau)=U_i(\tau)+2m_0\frac{d U_i}{d \tau}(\tau) \, .
$$
Using this new function, the system of ODEs \eqref{eq:1n} can be
rewritten in the following form: for $\tau \in (0,+\infty)$ and $i \in
\Z$,
$$
\left\{\begin{array}{l} \frac {dU_i}{d \tau}=\frac1{2 m_0} (\Xi_i-U_i) \\ \\
    \frac{d\Xi_i}{d\tau}=2\theta_{i+1}(U_{i+1}-U_i)-2\theta_i(U_i-U_{i-1})+2\sin
    (2\pi U_i)+2L+ \frac1{2m_0} (U_i-\Xi_i) \, .
\end{array}\right.
$$
We point out that, in compare with \cite{BM04,HQZ},  
our proof of the monotonicity of the system is simpler.
\bigskip

It is convenient to introduce the following notation
$$
\alpha_0 = \frac1{2 m_0} \, .
$$
\begin{rem}
  It would be also possible to consider more generally:
  $\Xi_i(\tau)=U_i(\tau) +\frac{1}{\alpha} \frac{dU_i}{d\tau}(\tau)$
  with $\frac{1}{\alpha}> m_0$. In order to simplify here the
  presentation, we choose $\alpha=1/(2m_0)$. Moreover, for the
  classical Frenkel-Kontorova model \eqref{eq:1}, the choice
  $\alpha=1/(2m_0)$ is optimal in the sense that the critical value
  $m_0^c$ for which the system is monotone is the best we can get.
\end{rem}

\subsection{General systems with $n$ types of particles}

More generally, we would like to study the generalized Frenkel-Kontorova
model \eqref{eq:fkg} with $n$ types of particles. 
In order to do so, let us consider a general sequence of functions
$v=(v_j(y))_{j\in\Z}$ satisfying
$$
v_{j+n}(y)= v_j(y+1) \, .
$$
For $m\in \N$, we set
$$
[v]_{j,m}(y)=(v_{j-m}(y),\dots,v_{j+m}(y)) \, .
$$
We are going to study a function
$$
(u,\xi)=((u_j(\tau,y))_{j\in \Z},(\xi_j(\tau,y))_{j\in \Z})
$$
satisfying the following system of equations: for all $(\tau,y)\in
(0,+\infty)\times \R$ and all $j\in\Z$,
\begin{equation}\label{eq:4n}
  \left\{\begin{array}{l}
      \left\{\begin{array}{l}
          (u_j)_\tau=\alpha_0( \xi_j-u_j)\\
          (\xi_j)_\tau=2F_j(\tau,[u(\tau,\cdot)]_{j,m})+\alpha_0 (u_j-\xi_j) \, ,
\end{array}\right.
\\ \\
\left\{ \begin{array}{l}
    u_{j+n}(\tau,y)= u_{j}(\tau,y+1)\\
    \xi_{j+n}(\tau,y)=\xi_j(\tau,y+1) \, .
\end{array}\right.
\end{array} \right.
\end{equation}
This system is referred to as \emph{the generalized Frenkel-Kontorova
  (FK for short) model}.  It is satisfied in the viscosity sense (see
Definition \ref{defi:2}). Moreover, we will consider viscosity
solutions which are possibly discontinuous.

Let us now make precise the assumptions on the functions
$F_j:\R \times \R^{2m+1}\to \R$ mapping $(\tau,V)$ to $F_j(\tau,V)$.
It is convenient to write $V\in \R^{2m+1}$ as $(V_{-m},\dots,V_m)$.
\begin{itemize}
\item[(A1)] {\bf (Regularity)}
$$
\left\{\begin{array}{l}
    F_j \mbox{ is continuous} \, , \\
    F_j \mbox{ is Lipschitz continuous in } V \mbox{ uniformly in }
    \tau \mbox{ and } j \, .
\end{array}\right.
$$
\item[(A2)] {\bf (Monotonicity in $V_i$, $i \neq 0$)}
$$
F_j(\tau,V_{-m},...,V_m) \mbox{ is non-decreasing in } V_i \quad
\mbox{ for } i \neq 0 \; .
$$
\end{itemize}

\begin{itemize}
\item[(A3)]{\bf (Monotonicity in $V_0$)}
$$
\alpha_0 + 2 \frac {\partial F_j}{\partial
    V_0}\ge 0 \quad \textrm{for all }j\in \Z\, .
$$
\end{itemize}
Keeping in mind the notation we chose above ($\alpha_0 = (2 m_0)^{-1}$), 
this assumption can be interpreted as follows: the mass has to be
small in comparison with the variations of the non-linearity, which
means that the system is sufficiently overdamped. This assumption
guarantees that $2F_j(\tau,V)+\alpha_0 V_0$ is non-decreasing in $V_0$
for all $j\in \Z$.
\begin{itemize}
\item[(A4)] {\bf (Periodicity)}
$$
\left\{\begin{array}{l}
    F_j(\tau,V_{-m}+1,  ...,V_m+1)=F_j(\tau,V_{-m},  ...,V_m) \, ,\\
    F_j(\tau+1, V)=F_j(\tau, V) \, .
\end{array}\right.
$$
 \item[(A5)] {\bf (Periodicity  of the type of particles)}
$$
F_{j+n}=F_j \quad \mbox{for all}\quad j\in \Z \, .
$$
\end{itemize}
When $n=1$, we explained in \cite{FIM2} that the system of ODEs can be
embedded into a single partial differential equation (more precisely,
in a single ordinary differential equation with a real parameter $x$).
Here, taking into account the ``$n$-periodicity'' of the indices
$j$, it can be embedded into $n$ coupled systems of equations.\medskip

 The next assumption allows us to guarantee that the ordering property of
the particles, \textit{i.e.}  $u_{j}\le u_{j+1}$, is preserved for all
time.
\begin{itemize}
\item[(A6)] {\bf (Ordering)} For all
  $(V_{-m},\dots,V_m,V_{m+1})\in\R^{2m+2}$ such that 
  $V_{i+1}\ge V_i$ for all $|i|\le m$, we have
$$
2F_{j+1}(\tau,V_{-m+1},\dots,V_{m+1})+ \a_0 V_1\ge 2F_{j}(\tau,V_{-m},\dots,V_m)+\a_0 V_0 \, .
$$
\end{itemize}
\begin{rem}
  If, for all $j\in \{1,\dots,n-1\}$, we have $F_{j+1}=F_j$ then assumption
  (A6) is a direct consequence of assumptions (A2) and (A3). Notice also
  that for $n \ge 1$, Condition~(A6') of Subsection~\ref{subsec}  does not
  allow us to take $\alpha_i= \frac1{2m_i}$ with different $m_i$'s. In
  particular, all the particles in our analysis have the same mass $m_0$.
\end{rem}
\begin{ex}
  We see that Assumptions~{\rm (A1)-(A5)} are in particular satisfied
  for the FK system (\ref{eq:1n}) with $n$ types of particles
  ($\theta_{n+j}=\theta_j$), $m=1$ and $F_j(\tau,V_{-1},V_0,V_1)=
  \theta_{j+1}(V_{1}-V_0) -\theta_j(V_0-V_{-1}) + \sin \left(2\pi
    V_0\right) + L$ for $\alpha_0 \ge
  2(\theta_j+\theta_{j+1})+4\pi$. To get {\rm (A6)} we have to assume
  furthemore that $\alpha_0 \ge 4\theta_j+4\pi$.
\end{ex}
We next rescale the generalized FK model: we consider for $\e>0$
$$\left\{\begin{array}{l}
    \displaystyle{u^\varepsilon_j(t,x)=\varepsilon u_j\left(\frac t \varepsilon,\frac x\varepsilon\right)}\\
    \\
    \displaystyle{\xi^\varepsilon_j(t,x)=\varepsilon \xi_j\left(\frac t \varepsilon,\frac x \varepsilon\right)} \, .
\end{array}\right.
$$
The function 
$(u^\eps,\xi^\e)=\left(\left(u^\eps_j(t,x)\right)_{j\in\Z},
  \left(\xi^\eps_j(t,x)\right)_{j\in\Z}\right)$ satisfies the following problem: for all $j\in\Z$,
$t >0$, $x \in \R$
\begin{equation}\label{eq:6n}
  \left\{\begin{array}{l}
      \left\{\begin{array}{l}
          (u^\eps_j)_t=\alpha_0 \frac {\xi^\e_j-u^\e_j}{\e} \\\\
          (\xi^\eps_j)_t=2F_j\left(\frac{t}{\eps},\left[\frac{u^\eps(t,\cdot)}{\eps}\right]_{j,m}\right)
          +\alpha_0\frac {u^\e_j-\xi^\e_j}\e \, .
        \end{array}\right.
      \\ \\
      \left\{\begin{array}{l}
          u_{j+n}^\eps(t,x)= u_{j}^\eps(t,x+\eps)\\
          \xi_{j+n}^\eps(t,x)= \xi_{j}^\eps(t,x+\eps)
        \end{array}\right.
\end{array}\right.
\end{equation}
We impose the following initial conditions
\begin{equation}\label{eq:6n-ic}
      \left\{\begin{array}{l}
          u^\eps_j(0,x)=u_0\left(x+\frac{j\eps}{n}\right)\\
          \xi^\eps_j(0,x)=\xi^\eps_0\left(x+\frac{j\eps}{n}\right) \, .
\end{array}\right.
\end{equation}
Finally, we assume that $u_0$ and $\xi_0^\e$ satisfy
\begin{itemize}
\item[\bf (A0)] {\bf (Gradient bound from below)} There exist $K_0>0$ and $M_0>0$ such that
\begin{eqnarray*}
  0< 1/K_0 \le (u_0)_x \le K_0 \quad \mbox{on}\quad \R \, , \\
  0< 1/K_0 \le (\xi_0^\e)_x \le K_0 \quad \mbox{on}\quad \R \, ,\\
  \|u_0-\xi_0^\e\|_\infty\le M_0\e \, .
\end{eqnarray*}
\end{itemize}
Then we have the following homogenization  result
\begin{theo}[\bf Homogenization of systems with $n$ types of
  particles] \label{th:3n} Assume that $(F_j)_j$ satisfies {\rm
    (A1)-(A6)}, and assume that the initial data $u_0, \xi_0^\e$
  satisfy {\rm(A0)}. Consider the solution $((u^\eps_j)_{j \in
    \Z},(\xi_j^\e)_{j \in \Z})$ of \eqref{eq:6n}-\eqref{eq:6n-ic}.
  Then, there exists a continuous function $\overline{F}:\R\mapsto \R$
  such that, for all integer $j \in \Z$, the functions $u^\eps_j$ and
  $\xi_j^\e$ converge uniformly on compact sets of $(0,+\infty)\times
  \R$ to the unique viscosity solution $u^0$ of (\ref{eq:3}).
\end{theo}
\begin{rem}
  The reader can be surprised by the fact that we obtain, at the
  limit, only one equation to describe the evolution of the system. In
  fact, this essentially comes from Assumption~{\rm (A6)} and the definition
  of $\xi_j^\e$. Indeed, it could be shown that assumption (A6)
  implies that the functions $u^\e$ and $\xi^\e$ are non-decreasing with
respect to $j$:  $u^\e_{j+1}\ge u^\e_j$ and $\xi^\e_{j+1}\ge
  \xi^\e_j$. Then, the system can be essentially sketched by only two
  equations (one for the evolution of $u$ and one for $\xi$). 
But by the ``microscopic definition'' of $\xi_j^\e$, we have
  $\xi_j^\e=u_j^\e+O(\e)$; hence only one equation is sufficient to
  describe the macroscopic evolution of all the system.
\end{rem}
\begin{rem}
  The case $m_0=0$ corresponds to $\a_0=+\infty$. In this case,
  $u^\e\equiv \xi^\e$ in \eqref{eq:6n} and Theorem \ref{th:3n} still
  holds true.
\end{rem}
We will explain in the next subsection how the non-linearity $\bar{F}$,
known as the effective Hamiltonian, is determined. We will see that this
has to do with the existence of solutions of \eqref{eq:6n},
\eqref{eq:6n-ic} of a specific form. They are constructed thanks to
functions referred to as hull functions.

\subsection{Hull functions}

In this subsection, we introduce the notion of {\it hull function} for
System~\eqref{eq:4n}.  More precisely, we look for special
functions $((h_j(\tau,z))_{j\in\Z}, (g_j(\tau,z))_{j\in \Z}$ such that
$(u_j(\tau,y),\xi_j(\tau,y)) = (h_j(\tau,py +\lambda \tau), g_j(\tau,py
+\lambda \tau))$ is a solution of \eqref{eq:4n} on $\Omega =
(-\infty,+\infty) \times \R=\R^2$. Here is a precise definition.
\begin{defi}[\bf Hull function for systems of $n$ types of
  particles] \label{defi:1n} ~ \newline 
Given $(F_j)_j$ satisfying {\rm
    (A1)-(A6)}, $p\in (0,+\infty)$ and a number $\lambda\in\R$, we say
  that a family of functions $((h_j)_j, (g_j)_j)$ is \emph{a hull
    function} for \eqref{eq:4n} if it satisfies for all
  $(\tau,z)\in\R^2$, $j\in \Z$
\begin{equation}\label{eq:10n}
\begin{array}{ll}
  \left\{\begin{array}{l}
      (h_j)_\tau+ \lambda (h_j)_z=\alpha_0(g_j-h_j) \\
      \\
      h_j(\tau+1,z)=h_j(\tau,z)\\
      h_j(\tau,z+1)=h_j(\tau,z)+1\\
      h_{j+n}(\tau,z)= h_{j}(\tau,z+p)   \\ 
      h_{j+1}(\tau,z)\ge h_{j}(\tau,z)\\
      (h_j)_z(\tau,z)\ge0\\
      \exists C \; {\rm s.t.}\; |h_j(\tau,z)-z|\le C 
\end{array}\right.
&
\left\{\begin{array}{l}
    (g_j)_\tau+ \lambda (g_j)_z=2F_j(\tau,[h(\tau,\cdot)]_{j,m}(z))+\alpha_0(h_j-g_j) \\
    \\
    g_j(\tau+1,z)=g_j(\tau,z)\\
    g_j(\tau,z+1)=g_j(\tau,z)+1\\
g_{j+n}(\tau,z)= g_{j}(\tau,z+p)   \\ 
 g_{j+1}(\tau,z)\ge g_{j}(\tau,z)\\
    (g_j)_z(\tau,z)\ge0\\
    \exists C \; {\rm s.t.}\; |g_j(\tau,z)-z|\le C \, .
\end{array}\right.
\end{array}
\end{equation}
In the case where the functions $(F_j)_j$ do not depend on $\tau$, we also require
that the hull function $((h_j)_j,(g_j)_j)$ is independent on $\tau$
and we denote it by $((h_j(z))_j,(g_j(z))_j)$.
\end{defi}
\begin{rem}\label{rem:osc-hull}
  The last line of \eqref{eq:10n} implies in particular that $\e
  h_j(\tau, \frac z \e)\to z$ and $\e g_j(\tau, \frac z \e)\to z$ as
  $\e \to 0$.
\end{rem}
Given $ p >0$, the following theorem explains how the effective
Hamiltonian $\overline{F} (p)$ is determined by an
existence/non-existence result of hull functions as $\lambda \in \R$
varies.
\begin{theo}[\bf Effective Hamiltonian and hull function] \label{th:2}
  Given $(F_j)_j$ satisfying {\rm (A1)-(A6)} and $p\in (0,+\infty)$,
  there exists a unique real number $\lambda$ for which there exists a hull
  function $((h_j)_j,(g_j)_j)$ (depending on $p$) satisfying
  \eqref{eq:10n}.  Moreover the real number $\lambda=\overline{F}(p)$, seen as a
  function of $p$, is continuous in $(0,+\infty)$.
\end{theo}

\subsection{Qualitative properties of the effective Hamiltonian}

We have moreover the following result
\begin{theo}[\bf Qualitative properties of
  $\overline{F}$] \label{th:4} Let $(F_j)_j$ satisfying
  {\rm(A1)-(A6)}. For any constant $L\in\R$, let $\overline{F}(L,p)$ denote the effective
  Hamiltonian given in Theorem \ref{th:2} for $p\in (0,+\infty)$,
  associated with $(F_j)_j$ replaced by $(L+F_j)_j$ .

  Then $(L,p) \mapsto \overline{F}(L,p)$ is continuous and we have the following properties\\
  {\bf (i)} {\bf (Bound)} we have
$$
|\overline{F}(L,p)-L|\le C_p \, .
$$
{\bf (ii)} {\bf (Monotonicity in $L$)}
$$
\overline{F}(L,p) \quad  \mbox{is non-decreasing in}\quad  L \, .
$$
\end{theo}

\subsection{Organization of the article}

In Section \ref{s2}, we give some useful results concerning viscosity
solutions for systems. In Section \ref{s3}, we prove the convergence
result assuming the existence of hull functions. The construction of
hull functions is given in Sections \ref{s4} and \ref{s5}. Finally,
Section \ref{s6} is devoted to the proof of the qualitative
properties of the effective Hamiltonian.

\subsection{Notation}

Given $r,R>0$, $t \in \R$ and $x \in \R$,  $Q_{r,R}(t,x)$ denotes 
the following neighbourhood of $(t,x)$
$$
Q_{r,R}(t,x) = (t-r,t+r)\times (x-R,x+R) \, .
$$

For $V=(V_1,\dots, V_N) \in \R^N$, $|V|_\infty$ denotes $\max_j
|V_j|$.  Given a family of functions $(v_j (\cdot))_{j \in \Z}$ and two
integers $j,m \in \Z$, $[v]_{j,m}$ denotes the function $(v_{j-m}
(\cdot), \dots, v_{j+m} (\cdot))$.

\section{Viscosity solutions}\label{s2}
\setcounter{equation}{0}

This section is devoted to the definition of viscosity solutions for
systems of equations such as \eqref{eq:4n}, \eqref{eq:6n} and
\eqref{eq:10n}. In order to construct hull functions when proving
Theorem~\ref{th:2}, we will also need to consider a perturbation of
\eqref{eq:4n} with linear plus bounded initial data. For all these
reasons, we define a viscosity solution for a generic equation whose
Hamiltonian $(G_j)_j$ satisfies proper assumptions.

Before making precise assumptions, definitions and crucial results
we will need later (such as stability, comparison principle,
existence), we refer the reader to the user's guide of Crandall,
Ishii, Lions \cite{CIL} and the book of Barles \cite{B} for an
introduction to viscosity solutions and \cite{CDE,LEN,I,IK91} and
references therein for results concerning viscosity solutions for
systems of weakly coupled partial differential equations.

\subsection{Main assumptions and definitions}
\label{subsec}

As we mentioned it before, we consider systems with 
general non-linearities $(G_j)_j$. Precisely, 
for $0< T\le +\infty$, we consider the following Cauchy problem: for
$j\in\Z$, $\tau >0$ and $y \in \R$,
\begin{equation}\label{eq:22n}
  \left\{\begin{array}{l}
      \left\{\begin{array}{l}
          (u_j)_\tau= \alpha_0 (\xi_j-u_j) \\
          (\xi_j)_\tau=G_j(\tau, [u(\tau,\cdot)]_{j,m},\xi_j,
          \inf_{y'\in\R}\left(\xi_j(\tau,y')-py'\right) 
          +py-\xi_j(\tau,y),(\xi_j)_y)
        \end{array} \right. 
      \\ \\
      \left\{ \begin{array}{l}          
          u_{j+n}(\tau,y)= u_{j}(\tau, y+1)\\
          \xi_{j+n}(\tau,y)=\xi_j(\tau,y+1)
        \end{array}  \right. 
\end{array}\right. 
\end{equation}
 submitted to the initial
conditions
\begin{equation}\label{eq:22n-ic}
      \left\{ \begin{array}{l}          
          u_j(0,y)=u_{0}(y+\frac j n):= u_{0,j}(y) \\
          \xi_j(0,y)=\xi_{0}(y+\frac j n):=\xi_{0,j}(y) \, .
          \end{array}  \right. 
\end{equation}
\begin{ex}
  The most important example we have in mind is the following one
$$
G_j(\tau,V_{-m}, \cdots, V_m
,r,a,q)=2F_j(\tau,V)+\alpha_0 (V_0-r)+\delta (a_0+a)q^+
$$
for some constants $\delta\ge 0$, $a_0,a,q\in \R$ and where $F_j$
appears in \eqref{eq:4n},\eqref{eq:6n}, \eqref{eq:10n}.
\end{ex}
In view of \eqref{eq:22n}, it is clear that in the case where $G_j$ effectively
depends on the variable $a$, solutions must be such that the infimum
of $\xi_j(\tau,y) - p \cdot y$ is finite for all time $\tau$. Hence, when $G_j$
does depend on $a$, we will only consider solutions $\xi_j$ satisfying for some $C_0(T)>0$:
for all $\tau \in [0,T)$ and all $y,y' \in \R$
\begin{equation}\label{eq:24}
|\xi_j(\tau,y+y')-\xi_j(\tau,y)-py'| \le C_0 \; . 
\end{equation}
When $T=+\infty $, we may assume that \eqref{eq:24} holds true for all time $T_0>0$ for 
a family of constants $C_0>0$. 

Since we have to solve a Cauchy problem, we have to assume that the initial
datum satisfies the assumption
\begin{itemize}
\item[(A0')] {\bf (Initial condition)}

\noindent
$(u_{0}, \xi_{0})$ satisfies (A0) (with $\e=1$); it also satisfies \eqref{eq:24} if $G_j$ depends on $a$ for some $j$. 
\end{itemize}

As far as the $(G_j)_j$'s are concerned, we make the following assumptions.
\begin{itemize}
\item[(A1')] {\bf (Regularity)}
\begin{itemize}
\item[(i)] $G_j$ is continuous.  
\item[(ii)]
For all $R>0$, there exists $L_0=L_0 (R)>0$ such that for all $\tau,V,W,r,s,a,q_1, q_2,j$, with $a\in [-R, R]$, we have
$$
|G_j(\tau,V,r,a,q_1) - G_j (\tau,W,s,a,q_2)| \le L_0|V-W|_\infty + L_0 |r-s|+L_0|q_1-q_2| \, .
$$ 
\item[(iii)]
There exists $L_1>0$ such that for all $V,a,b,\tau,r,q$,
$$
| G_j(\tau,V,r,a,q) - G_j (\tau,V,r,b,q)| \le L_1 |a-b| |q| \, .
$$
\end{itemize}
\item[(A2')] {\bf (Monotonicity in $V_i$, $i\ne 0$)}
$$
G_j(\tau,V_{-m},...,V_m,r,a,q) \mbox{ is non-decreasing in }  V_i  \mbox{ for } i \ne 0.
$$
\item[(A3')] {\bf (Monotonicity in $a$ and $V_0$)}
$$
G_j(\tau,V_{-m},...,V_m,r,a,q) \mbox{ is non-decreasing in } a \mbox{ and in } V_0.
$$
\item[(A4')] {\bf (Periodicity)} For all $(\tau,V,r,a,q)\in \R\times
  \R^{2m+1}\times\R\times \R\times \R$ and $j\in\{1,\dots,n\}$
$$
\left\{\begin{array}{l}
    G_j(\tau,V_{-m}+1,  ...,V_m+1,r+1,a,q)=G_j(\tau,V_{-m},  ...,V_m,r,a,q) \, ,\\
    G_j(\tau+1, V,r,a,q)=G_j(\tau, V,r,a,q) \, .
\end{array}\right.
$$
 \item[(A5')] {\bf (Periodicity  of the type of particles)}
$$
G_{j+n}=G_j \quad \mbox{for all}\quad j\in \Z \, .
$$
\item[(A6')] {\bf (Ordering)} For all $(V_{-m},\dots,V_m,V_{m+1})\in\R^{2m+2}$
such that  $\forall i, V_{i+1}\ge V_i$, we have
$$
G_{j+1}(\tau,V_{-m+1},\dots,V_{m+1},r,a,q)\ge G_{j}(\tau,V_{-m},\dots,V_m,r,a,q) \, .
$$

\end{itemize}
Finally, we recall the definition of the upper and lower semi-continuous envelopes, $u^*$ and $u_*$, 
of a locally bounded function $u$.
$$
u^*(\tau,y)=\limsup_{(t,x)\to (\tau,y)} u(t,x) \quad \text{ and } \quad 
u_*(\tau,y)=\liminf_{(t,x)\to (\tau,y)} u(t,x) \, .
$$
We can now define viscosity solutions for \eqref{eq:22n}.
\begin{defi}[\bf Viscosity solutions] \label{defi:2} Let $T>0$ and
  $u_{0}:\R\to \R$ and $\xi_{0}:\R\to \R$ be such that {\rm (A0')} is
  satisfied. For all $j$, consider locally bounded functions
  $u_j:\R^+\times \R\to \R$ and $\xi_j:\R^+\times \R\to \R$. We denote
  by $\Omega=(0,T]\times \R$.
\begin{itemize}
\item The function $((u_j)_j,(\xi_j)_j)$ is a \emph{sub-solution}
  (resp. a \emph{super-solution}) of \eqref{eq:22n} on $\Omega$ if
\eqref{eq:24} holds true for $\xi_j$ in the case where $G_j$  depends on $a$, and
$$
\forall j,n, \forall (\tau,y),  \quad 
u_{j+n}(\tau,y)=u_j(\tau,y+1), \quad \xi_{j+n}(\tau,y)=\xi_j(\tau,y+1)
$$ 
and for all $j\in \{1,\dots,n\}$, $u_j$ and $\xi_j$ are 
upper semi-continuous (resp. lower semi-continuous),  
and for all $(\tau,y)\in \Omega$ and any test function
  $\phi \in C^1(\Omega)$ such that $u_j-\phi$ attains a  local maximum
  (resp. a local minimum) at the point $(\tau,y)$, then we have
\begin{equation}\label{eq:201}
  \phi_\tau(\tau,y) \le
  \alpha_0 (\xi_j(\tau,y)-u_j(\tau,y) )
  \quad (\text{resp. } \ge ) 
\end{equation}
and for all $(\tau,y)\in \Omega$ and any test function
  $\phi \in C^1(\Omega)$ such that $\xi_j-\phi$ attains a  local maximum
  (resp. a local minimum) at the point $(\tau,y)$, then we have
\begin{equation}\label{eq:201bis}
  \phi_\tau(\tau,y)\le
  G_j(\tau,[u(\tau,\cdot)]_{j,m}(y),\xi_j(\tau,y),
  \inf_{y'\in\R}\left(\xi_j(\tau,y')-py'\right)+py-\xi_j(\tau,u),\phi_y(\tau,y))
\end{equation}
$$(\text{resp. } \ge ).$$
\item The function $((u_j)_j,(\xi_j)_j)$ is a
  \emph{sub-solution} (resp. \emph{super-solution}) of \eqref{eq:22n},\eqref{eq:22n-ic}
  if $((u_j)_j,(\xi_j)_j)$ is a sub-solution (resp. super-solution) on
  $\Omega$ and if it satisfies moreover for all $y \in
  \R, \; j\in \{1,\dots,n\}$
\begin{eqnarray*}
u_j(0,y)\le u_{0}(y+\frac j n) \quad (\text{resp. } \ge) \; , \\
\xi_j(0,y)\le \xi_{0}(y+\frac j n) \quad (\text{resp. } \ge) \; .
\end{eqnarray*}
\item A function $((u_j)_j,(\xi_j)_j)$ is a \emph{viscosity solution}
  of (\ref{eq:22n}) (resp. of \eqref{eq:22n},\eqref{eq:22n-ic}) if
  $((u^*_j)_j,(\xi^*_j)_j)$ is a sub-solution and
  $(((u_j)_*)_j,((\xi_j)_*)_j)$ is a super-solution of (\ref{eq:22n})
  (resp. of \eqref{eq:22n},\eqref{eq:22n-ic}).
\end{itemize}
\end{defi}
Sub- and super-solutions satisfy the following comparison principle
which is a key property of the equation.
\begin{pro}[{\bf Comparison principle}] \label{pro:3} ~ \newline Assume
  {\rm (A0')} and that $(G_j)_j$ satisfy {\rm (A1')-(A5')}.  Let
  $(u_j,\xi_j)$ (resp. $(v_j, \zeta_j)$) be a sub-solution (resp. a
  super-solution) of \eqref{eq:22n}, \eqref{eq:22n-ic} such that
  \eqref{eq:24} holds true for $\xi_j$ and $\zeta_j$ in the case where
  $G_j$ depends on $a$.  We also assume that there exists a constant
  $K>0$ such that for all $j\in \{1,\dots, n\}$ and $(t,x)\in
  [0,T]\times \R$, we have
\begin{equation}\label{eq:croissance}
u_j(t,x)\le u_{0,j}(x)+K(1+t),\quad \xi_j(t,x)\le \xi_{0,j}(x)+K(1+t)
\end{equation}
$$
\left({\rm resp. }-v_j(t,x)\le -u_{0,j}(x) +K(1+t),\quad -\zeta_j(t,x)\le -\xi_{0,j}(x)+K(1+t)\right) \, .
$$
If
$$
u_j(0,x)\le v_j(0,x)\quad {\rm and}\quad \xi_j(0,x)\le
\zeta_j(0,x)\quad \textrm{for all }j\in \Z,\; x\in \R \, ,
$$
then
$$
u_j(t,x)\le v_j(t,x)\quad {\rm and}\quad \xi_j(t,x)\le
\zeta_j(t,x)\quad \textrm{for all }j\in \Z,\; (t,x)\in [0,T]\times \R
\, .
$$
\end{pro}
\begin{rem}
  Even if it was not specified in \cite{FIM2}, the Lipschitz
  continuity in $q$ of $G_j$ is necessary to obtain a general
  comparison principle.
\end{rem}
\begin{proof}[Proof of Proposition~\ref{pro:3}]
  In view of assumption (A1')(i) and using the change of unknown
  functions $\bar u_j(t,x)=e^{-\lambda t} u_j(t,x)$ and $\bar
  \xi_j(t,x)=e^{-\lambda t}\xi_j(t,x)$, we classically assume, without
  loss of generality, that for all $r\ge s$
\begin{equation}\label{eq:L'}
G_j(\tau,V,r,a,q)-G_j(\tau,V,s,a,q)\le - L'(r-s)
\end{equation}
for $L'\ge L_0>0$.

We next define 
$$
M=\sup_{(t,x)\in (0,T)\times \R}\max_{j\in \{1,\dots,n\}}
\max\left(u_j(t,x)-v_j(t,x),\xi_j(t,x)-\zeta_j(t,x)\right) \, .
$$

The proof proceeds in several steps.

\noindent{\bf Step 1: The test function}\\
We argue by contradiction by assuming that $M>0$. Classically, we
duplicate the space variable by considering for $\e,\; \a$ and $\eta$
``small'' positive parameters, the functions
\begin{eqnarray*}
\varphi(t,x,y,j)=u_j(t,x)-v_j(t,y)- e^{At}\frac {|x-y|^2}{2\e}-\a|x|^2-\frac \eta {T-t} \\
\phi(t,x,y,j)=\xi_j(t,x)-\zeta_j(t,y)-e^{At}\frac {|x-y|^2}{2\e}-\a|x|^2-\frac \eta {T-t}
\end{eqnarray*}
where $A$ is a positive constant which will be chosen later.
We also consider 
$$
\Psi(t,x,y,j)=\max (\varphi(t,x,y,j), \phi(t,x,y,j))\, .
$$

Using Inequalities~\eqref{eq:croissance} and Assumption~(A0'), we get 
$$
u_j(t,x)-v_j(t,y)\le u_{0,j}(x)-u_{0,j}(y)+2K(1+T)\le K_0|x-y|+2K(1+T)
$$
and
$$
\xi_j(t,x)-\zeta_j(t,y)\le K_0|x-y|+2K(1+T) \, .
$$
We then deduce that
$$
\lim_{|x|,|y|\to \infty}\varphi(t,x,y,j)=\lim_{|x|,|y|\to \infty}\phi(t,x,y,j)=-\infty \, ,
$$ 
Using also the fact that $\varphi$ and $\phi$ are u.s.c,
we deduce that $\Psi$ reaches its maximum at some point $(\bar t,\bar
x, \bar y,\bar j)$.

Let us assume that $\Psi(\bar t, \bar x, \bar y,\bar j)=\phi(\bar t,
\bar x, \bar y,\bar j)$ (the other case being similar and even
simpler). Using the fact that $M>0$, we first remark that for $\a$ and
$\eta$ small enough, we have
$$
\Psi(\bar t, \bar x, \bar y,\bar j)=:M_{\e,\a,\eta}\ge \frac M 2>0 \, .
$$
In particular, 
$$
 \xi_{\bar j} (\bar t, \bar x) - \zeta_{\bar j} (\bar t, \bar y)
> 0 \, .
$$
\bigskip

\noindent{\bf Step 2: Viscosity inequalities for $\bar t>0$}\\
By duplicating the time variable and passing to the limit
\cite{CIL,B}, we classically get that there are real numbers $a,b
,\bar p \in \R$ such that
$$
a-b= \frac \eta {(T-\bar t)^2}+Ae^{A\bar t}\frac {|\bar x-\bar y|^2}{2\e},\quad \bar p= e^{A\bar t}\frac {\bar x -\bar y}\e 
$$
and
\begin{eqnarray*}
a &\le & G_{\bar j}(\bar t,[u(\bar t,\cdot)]_{\bar j,m}(\bar x),\xi_{\bar j}(\bar t,\bar x),
\inf(\xi_{\bar j}(\bar t,y')-p y')+p\bar x-\xi_{\bar j}(\bar t,\bar x),\bar p+2\a\bar x) \\
b &\ge & G_{\bar j}(\bar t,[v(\bar t,\cdot)]_{\bar j,m}(\bar y),\zeta_{\bar j}(\bar t,\bar y),
\inf(\zeta_{\bar j}(\bar t,y')-p y')+p\bar y-\zeta_{\bar j}(\bar t,\bar y),\bar p).
\end{eqnarray*}
Subtracting the two above inequalities, we  get
\begin{align}\label{eq:inequality1}
  \frac \eta {T^2}+Ae^{A\bar t}\frac {|\bar x-\bar y|^2}{2\e}\le&
  G_{\bar j}(\bar t,[u(\bar t,\cdot)]_{\bar j,m}(\bar x),\xi_{\bar
    j}(\bar t,\bar x),
  \inf(\xi_{\bar j}(\bar t,y')-p y')+p\bar x-\xi_{\bar j}(\bar t,\bar x),\bar p+2\a\bar x)\nonumber\\
  &- G_{\bar j}(\bar t,[v(\bar t,\cdot)]_{\bar j,m}(\bar
  y),\zeta_{\bar j}(\bar t,\bar y), \inf(\zeta_{\bar j}(\bar t,y')-p
  y')+p\bar y-\zeta_{\bar j}(\bar t,\bar y),\bar p)=:\Delta G_j \, .
\end{align}
\bigskip

\noindent{\bf Step 3: Estimate on $u_k(\bar t,\bar x)-v_k(\bar t,\bar y)$}\\
If $k\in \{1,\dots,n\}$, by the inequality $\varphi(\bar t, \bar
x,\bar y,k)\le \phi(\bar t,\bar x,\bar y,\bar j)$, we directly get
that
$$
u_k(\bar t,\bar x)-v_k(\bar t,\bar y)\le \xi_{\bar j}(\bar t,\bar
x)-\zeta_{\bar j}(\bar t, \bar y) \, .
$$
If $k\not \in \{1,\dots,n\}$, let us define $l_k\in \Z$ such that
$k-l_kn=\tilde k\in \{1,\dots,n\}$. By periodicity, we then have
\begin{align*}
  u_ k(\bar t,\bar x)-v_k(\bar t,\bar y)=&u_ {\tilde k+l_k n}(\bar t,\bar x)-v_ {\tilde k+l_k n}(\bar t,\bar y)\\
  =&u_ {\tilde k}(\bar t,\bar x+l_k)-v_ {\tilde k}(\bar t,\bar y+l_k)\\
  \le& \xi_{\bar j}(\bar t,\bar x)-\zeta_{\bar j}(\bar t,\bar y) -
  \a(|\bar x|^2-|\bar x+l_k|^2)
\end{align*}
where we have used the inequality $\varphi(\bar t,\bar x+l_k,\bar
y+l_k,\tilde k)\le \phi(\bar t,\bar x,\bar y,\bar j)$ to get the third
line.  Hence, for all $k\in \Z$ (and in particular for $k\in \{\bar j
-m,\dots,\bar j +m\}$), we finally deduce that
\begin{equation}\label{eq:borneu-v}
  u_ k(\bar t,\bar x)-v_k(\bar t,\bar y)\le  \xi_{\bar j}(\bar t,\bar x)-\zeta_{\bar j}(\bar t,\bar y) 
  + \a\left||\bar x|^2-|\bar x+l_k|^2\right|.
\end{equation}
\bigskip

\noindent{\bf Step 4: Estimate of $\Delta G_j$ in \eqref{eq:inequality1}}\\
Using successively \eqref{eq:borneu-v} and (A1')(ii), we obtain
\begin{eqnarray*}
 \Delta G_j
  &\le& G_{\bar j}\bigg(\bar t,\left[v(\bar t,\cdot) +\xi_{\bar j}(\bar t,\bar x)-\zeta_{\bar j}(\bar t,\bar y) 
+ \a\left||\bar x|^2-|\bar x+l_\cdot|^2\right|\right]_{\bar j,m}(\bar y),\xi_{\bar j}(\bar t,\bar x),\\
&&  \inf(\xi_{\bar j}(\bar t,y')-p y')+p\bar x-\xi_{\bar j}(\bar t,\bar x),\bar p+2\a\bar x\bigg)\\
&&  - G_{\bar j}\left(\bar t,[v(\bar t,\cdot)]_{\bar j,m}(\bar y),\zeta_{\bar j}(\bar t,\bar y),
\inf(\zeta_{\bar j}(\bar t,y')-p y')+p\bar y-\zeta_{\bar j}(\bar t,\bar y),\bar p\right)\\
  &\le& L_0 (\xi_{\bar j}(\bar t,\bar x)-\zeta_{\bar j}(\bar t,\bar y)) 
+L_0\a\max_{k\in \left\{\bar{j}-m,...,\bar{j}+m\right\}}\left||\bar x|^2-|\bar x+l_k|^2\right|\\
&&  +G_{\bar j}\left(\bar t,[v(\bar t,\cdot)]_{\bar j,m}(\bar y),\xi_{\bar j}(\bar t,\bar x),
\inf(\xi_{\bar j}(\bar t,y')-p y')+p\bar x-\xi_{\bar j}(\bar t,\bar x),\bar p+2\a\bar x\right)
\\
&&  - G_{\bar j}\left(\bar t,[v(\bar t,\cdot)]_{\bar j,m}(\bar y),\zeta_{\bar j}(\bar t, \bar y),
\inf(\zeta_{\bar j}(\bar t,y')-p y')+p\bar y-\zeta_{\bar j}(\bar t,\bar y),\bar p\right) \, .
\end{eqnarray*}
Now using successively \eqref{eq:L'} and (A1')(iii), we get
\begin{eqnarray}
\label{eq:long}
\Delta G_j  &\le& L_0 (\xi_{\bar j}(\bar t,\bar x)-\zeta_{\bar j}(\bar t,\bar y)) 
+L_0\a\max_{k\in \left\{\bar{j}-m,...,\bar{j}+m\right\}}
\left||\bar x|^2-|\bar x+l_k|^2\right|-L'(\xi_{\bar j}(\bar t,\bar x)-\zeta_{\bar j}(\bar t,\bar y))\\
  &&+G_{\bar j}\left(\bar t,[v(\bar t,\cdot)]_{\bar j,m}(\bar y),\zeta_{\bar j}(\bar t,\bar y),
\inf(\xi_{\bar j}(\bar t,y')-p y')+p\bar x-\xi_{\bar j}(\bar t,\bar x),\bar p+2\a\bar x\right)\nonumber\\
  &&- G_{\bar j}\left(\bar t,[v(\bar t,\cdot)]_{\bar j,m}(\bar y),\zeta_{\bar j}(\bar t, \bar y),
\inf(\zeta_{\bar j}(\bar t,y')-p y')+p\bar y-\zeta_{\bar j}(\bar t,\bar y),\bar p\right)\nonumber\\
\nonumber\\
  &\le&L\a\max_{k\in \left\{\bar{j}-m,...,\bar{j}+m\right\}}(2 |l_k\bar x|+l_k^2) \nonumber\\
  &&+L_1 \bigg( \inf(\xi_{\bar j}(\bar t,y')-p y')+p\bar x-\xi_{\bar j}(\bar t,\bar x)
-\inf(\zeta_{\bar j}(\bar t,y')-p y')-p\bar y+\zeta_{\bar j}(\bar t,\bar y)\bigg)^+|\bar p|\nonumber\\
  &&+G_{\bar j}\left(\bar t,[v(\bar t,\cdot)]_{\bar j,m}(\bar y),\zeta_{\bar j}(\bar t,\bar y),
 \inf(\xi_{\bar j}(\bar t,y')-p y')+p\bar y-\xi_{\bar j}(\bar t,\bar x),\bar p+2\a\bar x\right)\nonumber\\
  &&- G_{\bar j}\left(\bar t,[v(\bar t,\cdot)]_{\bar j,m}(\bar
    y),\zeta_{\bar j}(\bar t,\bar y),\inf(\xi_{\bar j}(\bar t,y')-p
    y')+p\bar y-\xi_{\bar j}(\bar t,\bar x),\bar p\right)\, . \nonumber
\end{eqnarray}
Using the fact that $\a|\bar x|\to 0$ as $\a\to 0$, we deduce that 
\begin{align*}
  &L\a\max_k(2 |l_k\bar x|+l_k^2) \\
  &+G_{\bar j}\left(\bar t,[v(\bar t,\cdot)]_{\bar j,m}(\bar y),\zeta_{\bar j}(\bar t,\bar y),
\inf(\xi_{\bar j}(\bar t,y')-p y')+p\bar y-\xi_{\bar j}(\bar t,\bar x),\bar p+2\a\bar x\right)\\
  &- G_{\bar j}\left(\bar t,[v(\bar t,\cdot)]_{\bar j,m}(\bar y),\zeta_{\bar j}(\bar t,\bar y),
\inf(\xi_{\bar j}(\bar t,y')-p y')+p\bar y-\xi_{\bar j}(\bar t,\bar y),\bar p\right)\\
  =&o_\a(1)
\end{align*}
where we have used \eqref{eq:24} to get a uniform bound $R >0$ for 
$\inf(\xi_{\bar j}(\bar t,y')-p y')+p\bar y-\xi_{\bar j}(\bar t,\bar y)$.
\medskip

\noindent{\bf Step 5: Passing to the limit}\\
Using the fact that $\phi(\bar t, y', y', \bar j)\le \phi(\bar t,\bar
x,\bar y,\bar j)$, we deduce that
$$
\xi_{\bar j}(\bar t,y')-\xi_{\bar j}(\bar t,\bar x)\le \zeta_{\bar j}(\bar t,y')-\zeta_{\bar j}(\bar t,\bar y)+\a|y'|^2.
$$
Combining this with the previous step, we get
\begin{eqnarray}\label{eq:inequality2}
  \frac \eta {T^2}+Ae^{A\bar t}\frac {|\bar x-\bar y|^2}{2\e}
  &\le& L_1\bigg(\inf(\zeta_{\bar j}(\bar t ,y')-p y'
    -\zeta_{\bar j}(\bar t,\bar y)+\a|y'|^2) \\
&&    -\inf(\zeta_{\bar j}(\bar t ,y')-p y' -\zeta_{\bar j}(\bar t,\bar y))\bigg)^+|\bar p| 
+ p(\bar x-\bar y)|\bar p|+ o_\a(1)\nonumber\\
  &\le&L_1\bigg(\inf(\zeta_{\bar j}(\bar t ,y')-p y'
    +\a|y'|^2)-\inf(\zeta_{\bar j}(\bar t ,y')-p y')\bigg)^+|\bar p| \nonumber \\
\nonumber
&&+
  pe^{A\bar t} \frac {|\bar x-\bar y|^2}{\e} + o_\a(1) \, .
\end{eqnarray}
Choosing $A=2p$, we finally get
$$
\frac \eta {T^2}\le o_\a(1)+\left(\inf(\zeta_{\bar j}(\bar t ,y')-p y'
  +\a|y'|^2)-\inf(\zeta_{\bar j}(\bar t ,y')-p y')\right)|\bar p|.
$$

Using the fact that for $\bar p=O(1)$ when $\alpha\to 0$ (in fact the
$O(1)$ depends on $\e$ which is fixed) and using classical arguments
about inf-convolution, we get that
$$
\left(\inf(\zeta_{\bar j}(\bar t ,y')-p y' +\a|y'|^2)-\inf(\zeta_{\bar
    j}(\bar t ,y')-p y')\right)|\bar p|=o_\a(1)
$$ 
and so
$$
\frac \eta {T^2}\le o_\a(1)
$$
which is a contradiction for $\a$ small enough. 
\medskip

\noindent{\bf Step 6: Case $\bar t=0$}\\
We assume that there exists a sequence $\e_n\to 0$ such that $\bar t=0$.
In this case, we have
$$
0<\frac M 2\le M_{\e_n,\a, \eta}\le \xi_0(\bar x)-\xi_0(\bar y)
-\frac {|\bar x-\bar y|^2}{2\e_n}-\a|x|^2\le  \xi_0(\bar x)-\xi_0(\bar y)
\le \|D\xi_0\|_{L^\infty}|\bar x-\bar y| \, .
$$
Using the fact that $|\bar x-\bar y|\to 0$ as $\e_n\to 0$ yields a
contradiction.
\end{proof} 
Let us now give a comparison principle on bounded sets. To this end,
for a given point $(\tau_0,y_0)\in (0,T)\times \R$ and for all
$r,R>0$, let us set
$$Q_{r,R}=(\tau_0-r,\tau_0+r)\times (y_0-R,y_0+R).$$
We then have the following result which proof is similar to the one of Proposition \ref{pro:3}
\begin{pro}[{\bf Comparison principle on bounded sets}]\label{pro:pc}
  ~\newline Assume {\rm (A1')-(A5')} and that $G_j(\tau, V, r, a, q)$
  does not depend on the variable $a$ for each $j$. Assume that
  $((u_j)_j,(\xi_j)_j)$ is a sub-solution (resp. $((v_j)_j,(\zeta)_j)$
  a super-solution) of \eqref{eq:22n} on the open set $Q_{r,R}\subset
  (0,T)\times \R$. Assume also that for all $j\in \{1,\dots, n\}$
$$
u_j\le v_j \quad {\rm and} \quad \xi_j\le \zeta_j\quad {\rm on}\; (\overline Q_{r,R+m}\backslash Q_{r,R}).
$$
Then $u_j\le v_j$ and $\xi_j\le \zeta_j$ on $Q_{r,R}$ for $j\in \{1,\dots, n\}$.  
\end{pro}
We now turn to the existence issue. Classically, we need to construct
barriers for \eqref{eq:22n}. In view of (A1')(ii) and (A4'), for $K_0$
given in (A0), the following quantity
\begin{equation}\label{eq:G0}
\overline G=\sup_{\tau\in \R, \; |q|\le K_0, \; j\in \{1,\dots, n\}}|G_j(\tau,0,0, 0, q)|
\end{equation}
is finite. Let us also denote $L_2:=L_1 K_0$. Hence, for all $\tau, a,
b, r\in \R$, $V\in \R^{2m+1}$, $q\in[-K_0, K_0]$ and $j\in \{1, \dots,
n\}$,
\begin{equation}\label{eq:K}
| G_j(\tau, V, r, a,q)-G_j(\tau, V, r, b, q)|\le L_2|a-b| .
\end{equation}
Then we have the following lemma
\begin{lem}[\bf Existence of barriers] \label{lem:1}
   Assume {\rm (A0')-(A5')}. There exists a constant
   $K_1>0$ such that
$$
((u^+_j(\tau, y))_j,(\xi^+_j(\tau,y))_j)= ((u_{0}(y+\frac j
n)+K_1\tau)_j,(\xi_{0}(y+\frac j n)+K_1\tau)_j)
$$
and
$$
((u^-_j(\tau, y))_j,(\xi^-_j(\tau,y))_j)= ((u_{0}(y+\frac j
n)-K_1\tau)_j,(\xi_{0}(y+\frac j n)-K_1\tau)_j)
$$
are respectively super and sub-solution of \eqref{eq:22n},
\eqref{eq:22n-ic} for all $T>0$.  Moreover, we can choose
\begin{equation}\label{eq:C}
  K_1=\max\left(L_2C_0+L_0\left(2+K_0\frac m n+M_0\right)+\overline G, \alpha_0 M_0\right)
\end{equation}
where $C_0$, $(K_0,M_0)$ and $\overline G$ are respectively given in
\eqref{eq:24}, {\rm (A0')} and \eqref{eq:G0}.
\end{lem}
\begin{proof}
  We prove that $((u^+_j(\tau, y))_j,(\xi^+_j(\tau,y))_j)$ is a
  super-solution of \eqref{eq:22n}, \eqref{eq:22n-ic}. In view of (A0)
  with $\eps =1$, we have for all $j\in \{1,\dots,n\}$
$$
\alpha_0 (\xi^+_j(\tau,y)-u^+_j(\tau,y) )=\alpha_0 (u_{0}(y+\frac j
n)-\xi_{0}(y+\frac j n)) \le\alpha_0 M_0 \le K_1
$$
and
\begin{align*}
  &G_j\bigg(\tau, [u^+(\tau,\cdot)]_{j,m}(y),\xi^+_j(\tau, y),
  \inf_{y'\in\R}\left(\xi^+_j(\tau,y')-py'\right)+py-\xi_j^+(\tau,y),(\xi^+_j)_y(\tau,y)\bigg)\\
  =&G_j\bigg(\tau, [u^+(\tau,\cdot)-\lfloor u^+_j(
  \tau,y)\rfloor]_{j,m}(y),\xi^+_j(\tau, y)
  -\lfloor u^+_j( \tau,y)\rfloor,\\
  &\quad \quad\inf_{y'\in\R}\left(\xi_{0}(y'+\frac j n)-py'\right)+py
  -\xi_{0}(y+\frac j n),(\xi_{0})_y(y+\frac j n)\bigg)\\
  \le& L_2 C_0+L_0+ L_0+ G_j\bigg(\tau, [u^+(\tau,\cdot)-u^+_j(
  \tau,y)]_{j,m}(y),\xi^+_j(\tau,y)
  - u^+_j( \tau,y),0,(\xi_{0})_y(y+\frac j n)\bigg)\\
  \le &
  L_2 C_0+L_0+L_0+ L_0 K_0 \frac m n +L_0M_0+ G_j\bigg(\tau, 0,\dots, 0, 0,0, (\xi_{0})_y(y+\frac j n)\bigg)\\
  \le &L_2 C_0+2L_0+ L_0 K_0 \frac m n+L_0M_0+\overline G
\end{align*}
where we have used the periodicity assumption~(A4') for the second
line, assumptions~(A0') and (A1')(ii) for the third line, the fact that
$|u_0(y+\frac{j+k}n)-u_0(y+\frac j n)|\le K_0 \frac mn$ for $|k|\le m$ and assumption (A0')
for the forth line and $|(\xi^+_j)_y|\le K_0$ for the last line.

When $G_j(\tau,V,r,a,q)$ is independent on $a$, we can simply choose
$L_2=0$.  This ends the proof of the Lemma.
\end{proof}
By applying Perron's method together with the comparison principle, we
immediately get from the existence of barriers the following result
\begin{theo}[\bf Existence and uniqueness for \eqref{eq:22n}]
  Assume {\rm (A0')-(A5')}.  Then there exists a unique solution
  $((u_j)_j, (\xi_j)_j)$ of \eqref{eq:22n}, \eqref{eq:22n-ic}.
  Moreover the functions $u_j, \xi_j$ are continuous for all $j$.
\end{theo}
We now claim that particles are ordered. 
\begin{pro}[\bf Ordering of the particles]\label{pro:croissancej} 
  Assume {\rm (A0')} and that the $(G_j)_j$'s satisfy {\rm
    (A1')-(A6')}. Let $(u_j,\xi_j)$ be a solution of
  \eqref{eq:22n}-\eqref{eq:22n-ic} such that \eqref{eq:24} holds true
  for $\xi_j$ if $G_j$ depends on $a$. Assume also that the $u_j$'s are
  Lipschitz continuous in space and let $L_u$ denote a common
  Lipschitz constant. Then $u_{j}$ and $\xi_{j}$ are non-decreasing
  with respect to $j$.
\end{pro}
\begin{proof}[Proof of Proposition \ref{pro:croissancej}]
  The idea of the proof is to define
  $(v_j,\zeta_j)=(u_{j+1},\xi_{j+1})$. In particular, we have
$$(v_j(0,y),\zeta_j(0,y))\ge (u_j(0,y),\xi_j(0,y)).$$
Moreover, $((v_j)_j,(\zeta_j)_j)$ is a solution of
$$
\left\{\begin{array}{l} 
\left\{\begin{array}{l} 
(v_j)_\tau=
    \alpha_0 (\zeta_j-v_j) ,\\
    (\zeta_j)_\tau=G_{j+1}(\tau, [v(\tau,\cdot)]_{j,m},\zeta_j,
    \inf_{y'\in\R}\left(\zeta_j(\tau,y')-py'\right)+py-\zeta_j(\tau,y),(\zeta_j)_y),\\
\end{array}\right. \\ \\
\left\{\begin{array}{l} 
    v_{j+n}(\tau,y)= v_{j}(\tau, y+1),\\
    \zeta_{j+n}(\tau,y)=\zeta_j(\tau,y+1)\\
    \end{array}\right. \\ \\
\left\{\begin{array}{l} 
    v_j(0,y)=u_{0}(y+\frac j n),\\
    \zeta_j(0,y)=\xi_{0}(y+\frac j n) \, .
\end{array}\right.
\end{array}\right.
$$
Now the goal is to obtain $u_j\le v_j$ and $\xi_j\le \zeta_j$. 
The arguments are essentially the same as those used in the proof of the comparison
principle. The main difference is that \eqref{eq:inequality1} is replaced with
\begin{align*}
  \frac \eta {T^2}+Ae^{A\bar t}\frac {|\bar x-\bar y|^2}{2\e}\le&
  G_{\bar j}(\bar t,[u(\bar t,\cdot)]_{\bar j,m}(\bar x),\xi_{\bar
    j}(\bar t,\bar x),
  \inf(\xi_{\bar j}(\bar t,y')-p y')+p\bar x-\xi_{\bar j}(\bar t,\bar x),\bar p+2\a\bar x)\nonumber\\
  &- G_{\bar j+1}(\bar t,[v(\bar t,\cdot)]_{\bar j,m}(\bar
  y),\zeta_{\bar j}(\bar t,\bar y),
  \inf(\zeta_{\bar j}(\bar t,y')-p y')+p\bar y-\zeta_{\bar j}(\bar t,\bar y),\bar p)\\
  \\
  \le& G_{\bar j}(\bar t,[u(\bar t,\cdot)]_{\bar j,m}(\bar
  y),\xi_{\bar j}(\bar t,\bar x),
  \inf(\xi_{\bar j}(\bar t,y')-p y')+p\bar x-\xi_{\bar j}(\bar t,\bar x),\bar p+2\a\bar x)\nonumber\\
  &- G_{\bar j+1}(\bar t,[v(\bar t,\cdot)]_{\bar j,m}(\bar
  y),\zeta_{\bar j}(\bar t,\bar y), \inf(\zeta_{\bar j}(\bar t,y')-p
  y')+p\bar y-\zeta_{\bar j}(\bar t,\bar y),\bar p) +L_0 L_u |\bar
  x-\bar y| \\
& =: \overline \Delta G_j 
\end{align*}
where we have used the Lipschitz continuity of $u$ and Assumption~(A1').

To obtain the desired contradiction, we have to estimate the right
hand side of this inequality. First, using Step~3 of the proof of the
comparison principle (with the same notation), we can define
$$
\delta:= \xi_{\bar j} (\bar t, \bar x)-\zeta_{\bar j}(\bar t,\bar y) +
L_u|\bar x-\bar y| + \a \max_{k\in\{\bar j-m,\dots, \bar j+m\}}(2|l_k
\bar x|+l_k^2)\ge 0
$$ 
such that for $k\in\{\bar j-m,\dots, \bar j+m\}$, we get from
\eqref{eq:borneu-v} the following estimate
\begin{equation}\label{eq:delta}
u_k(\bar t,\bar y)-v_k(\bar t,\bar y)\le \delta.
\end{equation}
Using Monotonicity Assumptions~(A2')-(A3') together with~(A1'), we get
\begin{eqnarray*}
\overline \Delta G_j 
& \le & G_{\bar j}(\bar t,[u (\bar t,\bar y)+ (\cdot-\bar{j})\delta]_{\bar{j},m},\xi_{\bar j}(\bar t,\bar x),
\inf(\xi_{\bar j}(\bar t,y')-p y')+p\bar x-\xi_{\bar j}(\bar t,\bar x),\bar p+2\a\bar x)\\
&&- G_{\bar j+1}(\bar t,[v(\bar t,\bar y)+ (\cdot+1)\delta ]_{\bar j,m},\zeta_{\bar j}(\bar t,\bar y),
\inf(\zeta_{\bar j}(\bar t,y')-p y')+p\bar y-\zeta_{\bar j}(\bar t,\bar y),\bar p) \\
&&+ L_0(2m+1)\delta +L_0 L_u |\bar x-\bar y|\, .
\end{eqnarray*}
Now we are going to use assumption (A6'). Remark first that we have for all $k\in
\{-m,m-1\}$
$$
v_{\bar j+k}(\bar t,\bar y)+ (m+k+1)\delta = u_{\bar j+k+1}(\bar t,\bar y)+ (m+k+1)\delta
$$
and for $k \in \{-m,\dots, m\}$, \eqref{eq:delta} yields
$$
u_{\bar j+k+1}(\bar t,\bar y)+ (m+k+1)\delta \ge u_{\bar j+k}(\bar t,\bar y) + (m+k)\delta \, .
$$
 Thus (A6') implies that 
\begin{multline} \label{eq:1000}
 G_{\bar j}(\bar t,[u(\bar t,\cdot)]_{\bar j,m}(\bar y),\xi_{\bar j}(\bar t,\bar x),
\inf(\xi_{\bar j}(\bar t,y')-p y')+p\bar x-\xi_{\bar j}(\bar t,\bar x),\bar p+2\a\bar x)
\\
\le  G_{\bar j+1}(\bar t,[v(\bar t,\bar y)+ (\cdot+1)\delta ]_{\bar j,m},\xi_{\bar j}(\bar t,\bar x),
\inf(\xi_{\bar j}(\bar t,y')-p y')+p\bar x-\xi_{\bar j}(\bar t,\bar x),\bar p+2\a\bar x) \, .
\end{multline}
Hence 
\begin{multline*}
\overline \Delta G_j  \le  G_{\bar j+1}(\bar t,[v(\bar t,\bar y)+ (\cdot+1)\delta ]_{\bar j,m},\xi_{\bar j}(\bar t,\bar x),
\inf(\xi_{\bar j}(\bar t,y')-p y')+p\bar x-\xi_{\bar j}(\bar t,\bar x),\bar p+2\a\bar x) \\
 - G_{\bar j+1}(\bar t,[v(\bar t,\bar y)+ (\cdot+1)\delta ]_{\bar j,m},\zeta_{\bar j}(\bar t,\bar y),
\inf(\zeta_{\bar j}(\bar t,y')-p y')+p\bar y-\zeta_{\bar j}(\bar t,\bar y),\bar p)\\
+ L_0(2m+1)(\xi_{\bar j} (\bar t, \bar x)-\zeta_{\bar j}(\bar t,\bar y)) +2 (m+1)L_0 L_u |\bar x-\bar y| + 
L_0 (2m+1) \alpha \max_{k\in\{\bar j-m,\dots, \bar j+m\}}(2|l_k
\bar x|+l_k^2)\, .
\end{multline*}
Now, to obtain the desired contradiction, it suffices to follow the
computation from  \eqref{eq:long}; in
particular, choose $L'\ge (2m+1) L_0$ in \eqref{eq:L'}. Then we obtain
$$\frac \eta {T^2}\le o_\a(1)+2(m+1)L_0 L_u|\bar x-\bar y|$$
which is absurd for $\a$ and $\e$ small enough (since $|\bar x-\bar y|\to 0$ as $\e\to 0$)
\end{proof}

\section{Convergence}\label{s3}
\setcounter{equation}{0}

This section is devoted to the proof of the main homogenization result
(Theorem~\ref{th:3n}). The proof relies on the existence of hull
functions (Theorem~\ref{th:2}) and qualitative properties of the
effective Hamiltonian (Theorem~\ref{th:4}). As a matter of fact, we
will use the existence of Lipschitz continuous sub- and super-hull
functions (see Proposition~\ref{pro:139}).  All these results are
proved in the next sections.  \medskip

We start with some preliminary results. Through a change of variables,
the following result is a straightforward corollary of Lemma
\ref{lem:1} and the comparison principle.
\begin{lem}[\bf Barriers uniform in $\eps$] \label{lem:2}
  Assume {\rm (A0)-(A5)}. Then there is a constant
  $C>0$, such that for all $\eps >0$, the solution $((u^\eps_j)_j,
  (\xi^\e_j))$ of \eqref{eq:6n}, \eqref{eq:6n-ic} satisfies for all $t>0$ and $x \in \R$
$$
|u_j^\eps(t,x)-u_{0}(x+\frac {j\e} n)|\le Ct \quad {\rm and }
\quad |\xi_j^\eps(t,x)-\xi_{0}^\e(x+\frac {j\e} n)|\le Ct.
$$
\end{lem}
We also have the following preliminary lemma. 
\begin{lem}[\bf $\eps$-bounds on the gradient] \label{lem:3} Assume
  {\rm (A0)-(A5)}. Then the solution
  $((u^\eps)_j,(\xi^\e_j)_j)$ of \eqref{eq:6n}, \eqref{eq:6n-ic} satisfies for all $t
  >0$, $x \in \R$, $z>0$ and $j\in \Z$
\begin{equation}\label{eq:28}
\eps \left\lfloor \frac{z}{\eps K_0}\right\rfloor \le
u^\eps_j(t,x+z)-u^\eps_j(t,x)\le \eps \left\lceil
\frac{zK_0}{\eps}\right\rceil 
\end{equation}
and 
$$\eps \left\lfloor \frac{z}{\eps K_0}\right\rfloor \le
\xi^\eps_j(t,x+z)-\xi^\eps_j(t,x)\le \eps \left\lceil
\frac{zK_0}{\eps}\right\rceil  \, .
$$
\end{lem}
\begin{rem}
  In particular we obtain that functions $u_j^\e(t,x)$ and
  $\xi_j^\e(t,x)$ are non-decreasing in $x$.
\end{rem}
\begin{proof}[Proof of Lemma \ref{lem:3}.]
We prove the bound from below (the proof is similar for the bound from
above). We first remark that (A0) implies that the initial condition satisfies for all $j\in \Z$
\begin{equation}\label{eq:u0}
u^\eps_j(0,x+z)=u_0(x+z+\frac {j\eps}n)\ge u_0(x+\frac {j\eps }n)+ z/K_0 \ge u^\eps_j(0,x)+ k \eps \quad
  \mbox{with}\quad k=\left\lfloor
  \frac{z}{\eps K_0}\right\rfloor
\end{equation}
and 
$$
\xi^\eps_j(0,x+z)\ge \xi^\eps_j(0,x)+ k \eps \, .
$$

From (A4), we know that for $\eps=1$, the equation is invariant by
addition of integers to solutions. After rescaling it,
Equation~\eqref{eq:6n} is invariant by addition of constants of the
form $k \eps$, $k \in \Z$. For this reason the solution of
\eqref{eq:6n} associated with initial data $((u^\eps_j(0,x) +k\eps
)_j, (\xi^\e_j(0,x)+k\e)_j)$ is $((u^\eps_j + k
\eps)_j,(\xi^\e_j+k\e)_j)$. Similarly the equation is invariant by
space translations. Therefore the solution with initial data
$((u^\eps_j(0,x+z))_j,(\xi^\e_j(0,x+z)_j)$ is $((u^\eps_j
(t,x+z))_j,(\xi^\e_j(t,x+z))_j)$. Finally, from \eqref{eq:u0} and the
comparison principle (Proposition \ref{pro:3}), we get
$$
u^\eps_j (t,x+z)\ge u^\eps_j (t,x)+ k \eps\quad {\rm and}\quad \xi^\eps_j (t,x+z)\ge \xi^\eps_j (t,x)+ k \eps
$$
which proves the bound from below. This ends the proof of the lemma.
\end{proof}

We now turn to the proof of  Theorem~\ref{th:3n}.\medskip

\begin{proof}[Proof of Theorem \ref{th:3n}.]
  We only have to prove the result for all $j\in \{1,\dots,n\}$. Indeed,
  using the fact that $u^\e_{j+n}(t,x)=u^\e_j(t,x+\e)$ and
  $\xi^\e_{j+n}(t,x)=\xi^\e_j(t,x+\e)$, we will get the complete
  result.

  For all $j\in \{1,\dots,n\}$, we introduce the following
  half-relaxed limits
$$
\overline{u}_j={\limsup_{\e\to0}}^*u^\e_j, \quad \overline{\xi}_j={\limsup_{\e\to0}}^*\xi^\e_j $$
$$
\underline{u}_j={\liminf_{\e\to0}}_*u^\e_j, \quad \underline{\xi}_j={\liminf_{\e\to0}}_*\xi^\e_j \, .$$
These functions are well defined thanks to Lemma~\ref{lem:2}.  We then
define
$$
\overline{v}=\max_{j\in  \{1,\dots,n\}} \max(\overline{u}_j, \overline \xi_j),
\quad \underline{v}=\min_{j\in  \{1,\dots,n\}}\min (\underline{u}_j,\underline \xi_j) \, .
$$
We get from Lemmas~\ref{lem:2} and \ref{lem:3} that both functions
$
w=\overline{v},\underline{v}$ satisfy for all $t>0$, $x,x' \in \R$, $x \le x'$
(recall that $\xi_0^\e\to u_0$ as $\e\to 0$)
\begin{eqnarray}
|w(t,x)-u_0(x)|&\le& Ct \, , \nonumber \\
K_0^{-1} |x-x'| &\le &w(t,x) - w(t,x') \le K_0 |x-x'| \, .\label{eq:grad-estim}
\end{eqnarray}
We are going to prove that $\overline{v}$ is a sub-solution of \eqref{eq:3}. 
Similarly, we can prove that $\underline{v}$ is a super-solution of the same equation. 
Therefore, from the comparison principle for \eqref{eq:3}, we
get that $u^0\le \underline{v}\le \overline{v}  \le u^0$. And then
$\overline{v}= \underline{v} =u^0$, which shows the expected convergence of
the full sequence $u^\eps_j$ and $\xi^\e_j$ towards $u^0$ for all $j\in  \{1,\dots,n\}$.
\medskip

We now prove in several steps that $\overline{v}$ is a sub-solution of
\eqref{eq:3}. We classically argue by contradiction: we assume that there exists
$(\ot,\ox)\in (0,+\infty)\times \R$ and a test function $\phi\in C^1$
such that
\begin{equation}\label{eq:31}
\left\{\begin{array}{lll}
\overline{v}(\ot,\ox) =\phi(\ot,\ox)&&\\
\overline{v}\le \phi & \quad \mbox{on}\quad Q_{r,2r}(\ot,\ox),&  \quad \mbox{with}
\quad r>0\\
\overline{v}\le \phi -2\eta & \quad \mbox{on}\quad
\overline{Q}_{r,2r}(\ot,\ox) \setminus Q_{r,r}(\ot,\ox),&  \quad \mbox{with}
\quad \eta >0\\
\phi_t(\ot,\ox)=\overline{F}(\phi_x(\ot,\ox)) + \theta, & \quad \mbox{with}
\quad \theta >0 \, .& 
\end{array}\right.   
\end{equation}
Let $p$ denote $\phi_x(\ot,\ox)$. From \eqref{eq:grad-estim}, we get
\begin{equation}\label{estimgrad}
0< 1/K_0 \le p \le K_0  \, .
\end{equation}
Combining Theorems~\ref{th:2} and \ref{th:4}, we get the existence of a hull function $((h_i)_{i}, (g_i)_i)$
associated with $p$ such that
$$
\lambda=\overline{F}(p) +\frac{\theta}{2}
=\overline{F}(\overline{L},p) \quad \mbox{with}\quad \overline{L}>0 \,.
$$
Indeed, we know from these results that the effective Hamiltonian is
non-decreasing in $L$, continuous and goes to $\pm \infty$ as $L \to
\pm \infty$.

We now apply the perturbed test function method introduced by Evans
\cite{E} in terms here of hull functions instead of
correctors. Precisely, let us consider the following twisted perturbed
test functions for $i\in \{1,\dots,n\}$
$$
\phi^\eps_i(t,x)=\eps
h_i\left(\frac{t}{\eps},\frac{\phi(t,x)}{\eps}\right),\quad
\psi^\e_i(t,x)=\e g_i
\left(\frac{t}{\eps},\frac{\phi(t,x)}{\eps}\right)\, .
$$
Here the test functions are twisted in the same way as in \cite{IM}.
We then define the family of perturbed test functions
$(\phi^\e_i)_{i\in \Z}, ((\psi^\e_i)_{i\in \Z})$ by using the
following relation
$$\phi^\eps_{i+kn}(t,x)=\phi^\eps_i(t,x+\e k), \quad \psi^\eps_{i+kn}(t,x)=\psi^\eps_i(t,x+\e k).$$
In order to get a contradiction, we first assume that the functions $h_i$ and
 $g_i$ are $C^1$ and continuous in $z$ uniformly in $\tau \in
\R, \;i\in \{1,\dots,n\}$.  In view of the third line of
\eqref{eq:10n}, we see that this implies that  $h_i$ and  $g_i$
are uniformly continuous in $z$ (uniformly in $\tau \in \R,\; i\in
\{1,\dots,n\}$). For simplicity, and since we will construct
approximate hull functions with such a (Lipschitz) regularity, we even assume that
 $h_i$ and  $g_i$ are globally Lipschitz continuous in $z$ (uniformly in
$\tau \in \R,\; i\in \{1,\dots,n\}$).  We will next see how to treat
the general case.  \medskip

\noindent {\bf Case 1: $h_i$ and $g_i$ are $C^1$ and globally Lipschitz continuous in $z$}
\medskip

\noindent {\bf Step 1.1: $((\phi^\eps_i)_i, (\psi^\e_i)_i)$ is  a super-solution of
  (\ref{eq:6n}) in a neighbourhood of $(\ot,\ox)$}

\noindent
When $h_i$ and $g_i$ are $C^1$,  it is sufficient to check directly the
super-solution property of $(\phi^\eps_i, \psi^\e_i)$ for $(t,x)\in Q_{r,r}(\ot,\ox)$. 
We begin by the equation satisfied by $\phi^\e_i$.
We have, with
$\tau=t /\eps$ and $z=\phi(t,x)/ \eps$,

\begin{align}
(\phi^\e_i)_t(t,x)=&(h_i)_\tau(\tau, z)+\phi_t(t,x) (h_i)_z(\tau, z) \nonumber\\
=&(\phi_t(t,x)-\lambda)(h_i)_z(\tau,z)+\alpha_0 (g_i(\tau,z)-h_i(\tau,z))\nonumber\\
=&\left(\phi_t(t,x)-\phi_t(\ot,\ox)+\frac \theta 2\right)(h_i)_z(\tau,z)
+\frac{\alpha_0}\eps (\psi^\e_i(t,x)-\phi^\e_i(t,x))
\nonumber\\
\ge&\frac{\alpha_0}\eps (\psi^\e_i(t,x)-\phi^\e_i(t,x))
\end{align}
where we have used the equation satisfied by $h_i$ to get the second
line and the non-negativity of $h_z$, the fact that $\theta >0$ and
the fact that $\phi$ is $C^1$, to get the last line on
$Q_{r,r}(\ot,\ox)$ for $r>0$ small enough.

We now turn to the equation satisfied by $\psi_i$. With the same notation, we have
\begin{align}\label{eq:first}
  &(\psi^\eps_i)_t(t,x) - 2 F_i\left(\tau,\left[\frac{\phi^\eps(t,\cdot)}{\eps}\right]_{i,m}(x)\right)
-\frac{\alpha_0}\eps (\phi^\e_i-\psi^\e_i)\\
  =&(g_i)_\tau(\tau,z) +\phi_t (t,x) (g_i)_z(\tau,z)
  - 2 F_i\left(\tau,\left[\frac{\phi^\eps(t,\cdot)}{\eps}\right]_{i,m}(x)\right) 
- \alpha_0  (h_i(\tau, z)-g_i(\tau, z))\nonumber\\
  =&(\phi_t(t,x)-\lambda)\  (g_i)_z (\tau,z) +2 \overline L 
+2\left( F_i\left(\tau, \left[h(\tau,\cdot)\right]_{i,m}(z)\right)
- F_i\left(\tau,\left[\frac{\phi^\eps(t,\cdot)}{\eps}\right]_{i,m}(x)\right)\right)\nonumber\\
  \ge & (\phi_t(t,x)-\lambda)\ (g_i)_z (\tau,z) +2 \overline L
  -2L_{F}\left|\left[h(\tau,\cdot)\right]_{i,m}(z) -
    \left[\frac{\phi^\eps(t,\cdot)}{\eps}\right]_{i,m}(x)\right|_\infty\nonumber
\end{align}
where we have used that Equation~\eqref{eq:10n} is satisfied by
$(g_i)_i$ to get the third line and (A1) to get the fourth one; here,
$L_F$ denotes the largest Lipschitz constants of the $F_i$'s (for
$i\in \{1,\dots,n\}$) with respect to $V$.

Let us next estimate, for $i\in \{1,\dots,n\},\; j \in \{-m,\dots,
m\}$ and $\eps>0$,
 \begin{eqnarray*}
 \I_{i,j}=h_{i+j}(\tau,z) - \frac{\phi^\eps_{i+j}(t,x)}{\eps}
 \end{eqnarray*}
If $i+j\in \{1,\dots,n\}$, then, by definition of $\phi_{i+j}$, we have
$$
\I_{i,j}=h_{i+j}\left(\frac t \e,\frac {\phi(t,x)}\e\right)- \frac{\phi^\eps_{i+j}(t,x)}{\eps}=0.
$$
If $i+j\not \in \{1,\dots,n\}$, let us define $l$ such that $1\le
i+j-ln\le n$. We then have
\begin{align*}
  \I_{i,j}=&h_{i+j-ln}(\tau, z+lp)-\frac {\phi^\e_{i+j-ln}(t,x+\e l)}\e\\
  =&h_{i+j-ln}\left(\frac t \e, \frac {\phi(t,x)} \e+lp\right)-h_{i+j-ln}\left(\frac t\e,\frac{\phi(t,x+\e l)}\e \right)\\
  =&h_{i+j-ln}\left(\frac t \e, \frac {\phi(t,x)}
    \e+lp\right)-h_{i+j-ln}\left(\frac t\e,\frac{\phi(t,x)}\e +lp +
    o_r(1) \right)
\end{align*}
where $o_r(1)$ only depends on the modulus of continuity of $\phi_x$
on $Q_{r,r} (\ot,\ox)$ (for $\e$ small enough such that $\e l\le r$ with $l$ uniformly bounded and then $(t, x+\e l)\in Q_{r,2r}(\ot,\ox)$).  Hence, if $h_i$ are Lipschitz continuous with
respect to $z$ uniformly in $\tau$ and $i$, we conclude that we can
choose $\eps$ small enough so that
 \begin{equation}\label{eq:second}
   \overline{L} 
   - L_F \left| \left[h(\tau,\cdot)\right]_{i,m}(z)  
     -  \left[\frac{\phi^\eps(t,\cdot)}{\eps}\right]_{i,m}(x) \right|_\infty \ge 0\, .
 \end{equation}
 Combining \eqref{eq:first} and \eqref{eq:second}, we obtain
\begin{align*}
  (\psi^\eps_i)_t(t,x) -
  2F_i\left(\tau,\left[\frac{\phi^\eps(t,x)}{\eps}\right]_{i,m}(x)\right)+
  \frac{\alpha_0}{\e} (\phi^\e_i-\psi^\e_i)
  \ge &  \left(\phi_t(t,x)-\lambda \right) \  (g_i)_z (\tau,z)\\
  \ge &  \left(\frac \theta 2+\phi_t(t,x)-\phi_t(\ot,\ox)\right) \  (g_i)_z (\tau,z) \\
  = &\left(\frac \theta 2+ o_r (1)\right) \ (g_i)_z (\tau,z) \ge 0 \,
  .
\end{align*}
We used the non-negativity of $(g_i)_z$, the fact that $\theta >0$ and
again the fact that $\phi$ is $C^1$, to get the result on
$Q_{r,r}(\ot,\ox)$ for $r>0$ small enough.  Therefore, when the $h_i$
and $g_i$ are $C^1$ and Lipschitz continuous on $z$ uniformly in
$\tau$ and $i$, $((\phi^\eps_i)_i, (\psi^\e_i)_i)$ is a viscosity
super-solution of (\ref{eq:6n}) on $Q_{r,r}(\ot,\ox)$.  \medskip

\noindent {\bf Step 1.2: getting the contradiction}

\noindent
By construction (see Remark~\ref{rem:osc-hull}), we have $\phi^\eps_i \to \phi$ and $\psi^\e_i\to
\phi$ as $\eps\to 0$ for all $i\in \{1,\dots,n\}$, and therefore from
the fact that $\overline u_j \le \bar v \le \phi -2\eta$ on
$\overline{Q}_{r,2r}(\ot,\ox) \setminus Q_{r,r}(\ot,\ox)$ (see
(\ref{eq:31})), we get for $\eps$ small enough
$${u}^\eps_i\le \phi^\eps_i -\eta \le  \phi^\eps_i -\eps k_\eps\quad \mbox{on}\quad 
\overline{Q}_{r,2r}(\ot,\ox) \setminus Q_{r,r}(\ot,\ox)$$ with the
integer
$$k_\eps=\lfloor \eta/\eps \rfloor \, .
$$ 
In the same way, we have
$$
{\xi}^\eps_i\le \psi^\eps_i -\eta \le  \psi^\eps_i -\eps k_\eps\quad \mbox{on}\quad 
\overline{Q}_{r,2r}(\ot,\ox) \setminus Q_{r,r}(\ot,\ox) \, .
$$ 
Therefore,
for $m\eps \le r$, we can apply the comparison principle on bounded
sets 
to get
\begin{equation}\label{eq:104}
{u}^\eps_i\le \phi^\eps_i  -\eps k_\eps , \quad {\xi}^\eps_i\le \psi^\eps_i  -\eps k_\eps \quad \mbox{on}\quad
{Q}_{r,r}(\ot,\ox) \, .
\end{equation}
Passing to the limit as $\eps$ goes to zero, we get
$$
\overline{u}_i \le \phi -\eta, \quad \overline{\xi}_i \le \phi -\eta
\quad \mbox{on}\quad {Q}_{r,r}(\ot,\ox)
$$
which implies that
$$
\overline{v} \le \phi -\eta  \quad \mbox{on}\quad {Q}_{r,r}(\ot,\ox) .
$$
This gives a contradiction with $\overline{v}(\ot,\ox) =\phi(\ot,\ox)$
in (\ref{eq:31}).  Therefore $\overline{v}$ is a sub-solution of
(\ref{eq:3}) on $(0,+\infty)\times \R$ and we get that $u^\e_j$ and
$\xi^\e_j$ converges locally uniformly to $u^0$ for $j\in
\{1,\dots,n\}$. This ends the proof of the theorem.\medskip

\noindent {\bf Case 2: general case for $h$}

\noindent
In the general case, we can not check by a direct computation that
$((\phi^\eps_i)_i, (\psi^\e_i)_i)$ is a super-solution on
$Q_{r,r}(\ot,\ox)$. The difficulty is due to the fact that the $h_i$
and the $g_i$ may not be Lipschitz continuous in the variable $z$.

This kind of difficulties were overcome in \cite{IM} by using
Lipschitz super-hull functions, \textit{i.e.} functions satisfying
\eqref{eq:10n}, except that the function is only a super-solution of
the equation appearing in the first line. Indeed, it is clear from the
previous computations that it is enough to conclude.  In \cite{IM},
such regular super-hull functions (as a matter of fact, regular
super-correctors) were built as exact solutions of an approximate
Hamilton-Jacobi equation.  Moreover this Lipschitz continuous hull
function is a super-solution for the exact Hamiltonian with a slightly
bigger $\lambda$.

Here we conclude using a similar result, namely Proposition \ref{pro:139}.
Notice that in Proposition \ref{pro:139} $h_i$ and $g_i$ are only Lipschitz
continuous and not $C^1$. This is not a restriction, because the result of
Step 1.1 can be checked in the viscosity sense using test function (see
\cite{E} for further details). Comparing with \cite{IM}, notice that we
do not have to introduce an additional dimension because here $p>0$ (see
\eqref{estimgrad}).  This ends the proof of the theorem.
\end{proof}

\section{Ergodicity and construction of  hull functions}\label{s4}
\setcounter{equation}{0}
In this section, we first study the ergodicity of the equation~\eqref{eq:22n}
by studying the associated Cauchy problem (Subsection~\ref{subsec:ergo}). 
We then construct hull functions (Subsection~\ref{subsec:hull}). 

\subsection{Ergodicity}
\label{subsec:ergo}

In this subsection, we study the Cauchy problem associated with
\eqref{eq:22n} with 
\begin{equation}\label{defi:gdelta1}
G_j(\tau,V,r,a,q)=G_j^\delta(\tau,V,r,a,q)=2F_j(\tau,V)+\alpha_0 (V_0 -r) +\delta(a_0+a)q^+ 
\end{equation}
with $\delta \ge 0$, $a_0 \in \R$
and with initial data $y \mapsto py$. We prove that there
exists a real number $\lambda$ (called the ``slope in time'' or
``rotation number'') such that the solution $(u_j,\xi_j)$ stays at a
finite distance of the linear function $\lambda \tau + py$. We also
estimate this distance and give qualitative properties of the
solution.

We begin by a regularity result concerning the solution of  \eqref{eq:22n}.
\begin{pro}[\bf Bound on the gradient]\label{pro:130}
  Assume {\rm (A1)-(A5)} and $p>0$. Let $\d>0$, $a_0\in \R$ and $(u_j,\xi_j)_j$ be the solution  of
  \eqref{eq:22n}, \eqref{eq:22n-ic} with $G_j = G^\delta_j$ defined by
  \eqref{defi:gdelta1} and $u_0 (y) = py$. Assume that \eqref{eq:24} holds true for $\xi_j$. Then $(u_j,\xi_j)_j$ satisfies
\begin{equation}\label{eq:134}
0\le (u_j)_y \le p + \frac{2L_F}\delta \quad \text{ and } \quad
0\le (\xi_j)_y \le p + \frac{2L_F}\delta 
\end{equation}
where $L_F$ denotes the largest Lipschitz constant of the $F_i$'s for
$i=1,\dots, n$. 
\end{pro}
\begin{proof}
We first show that $u_j$ and $\xi_j$ are non-decreasing with respect to $y$. 
Since the equation \eqref{eq:22n} is invariant by translations in $y$
and using the fact that for all $b\ge 0$, we have
$$
u_0(y+b+\frac j n)\ge u_0(y+\frac j n) \, .
$$ 
We deduce from the comparison principle that
$$
u_j(\tau,y+b)\ge u_j(\tau,y)\quad {\rm and}\quad \xi_j(\tau,y+b)\ge \xi_j(\tau,y)
$$
which shows that $u_j$ and $\xi_j$ are non-decreasing in $y$.\medskip

We now explain how to get the Lipschitz estimate. We would like to prove that $\overline M \le 0$ where
\begin{align*}
\overline M = \sup_{\tau \in (0,T), x, y \in \R, j \in \{1, \dots, n\}} 
\max \bigg\{& u_j(\tau,x) - u_j(\tau,y) - L |x-y|- \frac{\eta}{T-\tau}-\a|x|^2, \\
& \; \xi_j(\tau,x) 
- \xi_j(\tau,y) - L |x-y| - \frac{\eta}{T-\tau}-\a|x|^2\bigg\}
\end{align*}
as soon as $L > p + \frac{2L_F}{\delta}>0$ for any $\eta,\a>0$. We
argue by contradiction by assuming that $\overline M>0$ for such an
$L$. We next exhibit a contradiction. The supremum defining $\overline
M$ is attained since $\xi_j$ satisfies \eqref{eq:24} and $u_j$ can be
explicitly computed.

\paragraph{Case 1.} Assume that the supremum is attained for the function $u_j$ at 
$\tau \in [0,T)$, $j \in \{1,\dots ,n\}$,
$x, y \in \R$. Since we have by assumption $\overline M>0$, this implies
that $\tau >0$, $x \neq y$. Hence we can obtain the two following viscosity inequalities 
(by doubling the time variable and passing to the limit)
\begin{eqnarray*}
a & \le & \alpha_0 (\xi_j (\tau, x) - u_j (\tau,x)) \\
b & \ge & \alpha_0 (\xi_j (\tau, y) - u_j (\tau,y)) 
\end{eqnarray*}
with $a-b = \frac{\eta}{(T-\tau)^2}$. Subtracting these inequalities, we obtain
$$
\frac{\eta}{(T-\tau)^2} \le \alpha_0 ( \{ \xi_j (\tau, x)
-\xi_j (\tau, y) \} - \{ u_j (\tau,x)- u_j (\tau,y) \})\le 0\, .
$$
We thus get $\eta \le 0$ which is a contradiction in Case 1. 

\paragraph{Case 2.} Assume next that the supremum is attained for the function $\xi_j$. 
By using the same notation and by arguing similarly, we obtain the following inequality
\begin{eqnarray*}
  \frac{\eta}{(T-\tau)^2}& \le &
  2 F_j (\tau, u_{j-m}(\tau,x), \dots , u_{j+m}(\tau,x)) 
  -   2F_j (\tau, u_{j-m}(\tau,y), \dots , u_{j+m}(\tau,y)) \\
  && +\alpha_0 (\{ u_j (\tau,x) - u_j (\tau,y) \} -\{ \xi_j (\tau,x) - \xi_j (\tau,y)\} ) \\
  &&+  \delta \{ p (x-y) - (\xi_j (\tau,x) -\xi_j (\tau,y)) \} L \mathrm{sign}^+ \; (x-y) 
  + 2\a \delta(a_0+C_0)|x|
\end{eqnarray*}
where $\mathrm{sign}^+$ is the Heaviside function and where we have
used \eqref{eq:24}. We now use
\begin{itemize}
\item the fact that the supremum is attained for the function $\xi_j$ 
\item the fact that $\xi_j (\tau, x) > \xi_j (\tau, y)$ implies that $x >y$
  (remember that we already proved that $\xi_j$ is non-decreasing with
  respect to $y$)
\item Assumption~(A1);  in the following, $L_F$ still denotes de largest Lipschitz constants
of the $F_j$'s with respect to $V$; 
\item the fact that $\a \delta(a_0+C_0)|x|=o_\a(1)$ 
\end{itemize}
in order to get from the previous inequality the following one
\begin{eqnarray*}
\frac{\eta}{(T-\tau)^2}& \le & 2 L_F \sup_{l \in \{-m, \dots, m\}} | u_{j+l} (\tau,x) -u_{j+l} (\tau,y)|  
 + \delta p L |x-y| - L \delta (\xi_j (\tau,x) - \xi_j (\tau,y) )+o_\a(1)\ . 
\end{eqnarray*}
Using the same computation as the one of the proof of Proposition
\ref{pro:3} Step 3, we get
$$ 
\sup_{l \in \{-m, \dots, m\}} | u_{j+l} (\tau,x) -u_{j+l} (\tau,y)|
= \sup_{l \in \{-m, \dots, m\}} ( u_{j+l} (\tau,x) -u_{j+l} (\tau,y) ) 
\le \xi_j (\tau,x) - \xi_j(\tau,y)+ C \a (1+|x|)
$$
where $C$ is a constant.  Since $C \a (1+|x|)=o_\a(1)$ and $\overline
M>0$, we finally deduce that
$$
\frac{\eta}{T^2}\le2 L_F (\xi_j (\tau,x) - \xi_j(\tau,y)) + \delta p (\xi_j (\tau,x) - \xi_j(\tau,y)) 
- L \delta (\xi_j (\tau,x) - \xi_j (\tau,y) )+o_\a(1)
$$ 
For $\alpha $ small enough, it is now sufficient to use once again
that $\xi_j (\tau,x) > \xi_j (\tau,y)$ and the fact that $L > p +
\frac{2L_F}\delta$ in order to get the desired contradiction in Case
2.  The proof is now complete.
\end{proof}
\medskip


We now claim that particles are ordered. 
\begin{pro}[\bf Ordering of the particles]\label{pro:croissancej-delta} 
  Assume {\rm (A0')}, {\rm (A1)-(A6)} and let $\d\ge0$, $a_0\in \R$ and
  $(u_j^\d,\xi_j^\d)_j$ be the solution of \eqref{eq:22n},
  \eqref{eq:22n-ic} with $G_j = G^\delta_j$ defined by
  \eqref{defi:gdelta1}. Assume that \eqref{eq:24} holds true for
  $\xi_j$ if $\d>0$. Then $u_{j}^\d$ and $\xi_{j}^\d$ are
  non-decreasing with respect to $j$.
\end{pro}

\begin{proof}
  If $\d>0$, the results is a straightforward consequence of
  Propositions \ref{pro:croissancej} and \ref{pro:130}. If $\d=0$, the
  result is obtained by stability of viscosity solution
  (\textit{i.e.} $u^\d_j\to u^0_j$ and $\xi_j^\d\to \xi_j^0$ as $\d\to 0$).
\end{proof}

\begin{pro}[\bf Ergodicity] \label{pro:11} Let $0\le \delta\le 1$ and
  $a_0\in \R$.  Assume {\rm (A0)-(A6)} and let $(u_j, \xi_j)_j$ be a
  solution of \eqref{eq:22n}, \eqref{eq:22n-ic} with $G_j$ defined in
  \eqref{defi:gdelta1} and with initial data $u_{0}(y)=\xi_{0}(y)=py$
  with some $p>0$. Then there exists $\lambda\in\R$ such that for all
  $(\tau,y)\in [0,+\infty )\times \R$, $j\in \{1,\dots,n\}$
\begin{equation}\label{eq:41}
  |u_j(\tau,y)-py-\lambda \tau|\le C_3
  \quad \text{ and } \quad |\xi_j(\tau,y)-py-\lambda \tau|\le C_3
\end{equation} 
and 
\begin{equation}\label{eq:140ante}
|\lambda| \le C_4 
\end{equation}
where
\begin{eqnarray}\label{eq:C4}
   C_3&=&13+\frac{6C_4}{\alpha_0}+7p+2K_1\nonumber \\
   C_4&=& \max\left(\alpha_0 M_0, L_F(2+p(m+n)) +\sup_\tau |F(\tau,0,\dots, 0)|+(p/2+L_F)(a_0+C_0)\right)
\end{eqnarray}
(where $a_0$ is chosen equal to zero for $\delta=0$).
Moreover we have for all $\tau \ge 0$, $y,y' \in \R$, $j\in
\{1,\dots,n\}$
\begin{equation}\label{eq:45}
  \left\{\begin{array}{l}
    u_j(\tau,y+1/p)=u_j(\tau,y) +1 \\
    (u_j)_y(\tau,y)\ge 0 \\
    |u_j(\tau,y+y')-u_j(\tau,y)-py'|\le 1\\
    u_{j+1}(\tau,y)\ge u_j(\tau,y) 
\end{array}\right.
\quad
  \left\{\begin{array}{lll}
\xi_j(\tau,y+1/p)=\xi_j(\tau,y) +1 \\
(\xi_j)_y(\tau,y)\ge 0\\
 |\xi_j(\tau,y+y')-\xi_j(\tau,y)-py'|\le 1 \\
\xi_{j+1}(\tau,y)\ge \xi_j(\tau,y) \ .
\end{array}\right.
\end{equation}
\end{pro}

In order to prove Proposition \ref{pro:11}, we will need the following classical lemma
from ergodic theory (see for instance \cite{kato}). 
\begin{lem}\label{lem:ergo}
Consider $\Lambda : \R^+ \to \R$ a continuous function which is sub-additive,
that is to say: for all $t,s \ge 0$, 
$$
\Lambda (t+s) \le \Lambda (t) + \Lambda (s) \, .
$$
Then $\frac{\Lambda (t)}{t}$ has a limit $l$ as $t \to + \infty$ and 
$$
l = \inf_{t>0} \frac{\Lambda (t)}t \, .
$$
\end{lem} 
We now turn to the proof of Proposition~\ref{pro:11}.
\begin{proof}[Proof of Proposition \ref{pro:11}] We perform the proof in three steps.
 We first recall that the fact
that $u_j$ and $\xi_j$ are non-decreasing in $y$ and $j$ follows from
Propositions~\ref{pro:130} and  \ref{pro:croissancej-delta}.
\medskip

\noindent {\bf Step 1: control of the space oscillations.} We are going to prove
the following estimate.
\begin{lem}
For all $\tau >0$, all $y,y' \in \R$ and all $j \in \{1, \dots, n\}$, 
\begin{equation}\label{eq:sp-osc}
|u_j(\tau,y+y')-u_j(\tau,y)-py'|\le 1\quad{\rm and}\quad |\xi_j(\tau,y+y')-\xi_j(\tau,y)-py'|\le 1 \, .
\end{equation}
\end{lem}
\begin{proof}
We have
$$
u_j(0,y+1/p)=\xi_j(0,y+1/p)=\xi_j(0,y) +1 =u_j(0,y) +1 \, .
$$
Therefore from the comparison principle and from the integer periodicity
 of the Hamiltonian (see (A3')), we get that
$$
u_j(\tau,y+1/p)=u_j(\tau,y) +1 \quad {\rm and}\quad \xi_j(\tau,y+1/p)=\xi_j(\tau,y) +1 \, .
$$
Since $u_j(\tau,y)$ is non-decreasing in $y$, we deduce that for all  $b\in [0,1/p]$
$$
0\le u_j(\tau,b)-u_j(\tau,0)\le 1 
$$
Let now $y\in\R$, that we write $py=k+a$ with $k\in\Z$ and $a\in [0,1)$.
Then we have
$$
u_j(\tau,y)-u_j(\tau,0)=k + u_j(\tau,a/p)-u_j(\tau,0)
$$
which implies, for some $b\in [0,1/p)$, 
$$
u_j(\tau,y)-u_j(\tau,0)-py=-a +u_j(\tau,b)-u_j(\tau,0)
$$
and then for all $\tau >0$ and all $y \in \R$, 
$$
|u_j(\tau,y)-u_j(\tau,0)-py|\le 1 \, .
$$
In the same way, we get 
$$
|\xi_j(\tau,y)-\xi_j(\tau,0)-py|\le 1 \, .
$$
Finally, we obtain \eqref{eq:sp-osc} by using the invariance by
translations in $y$ of the problem.
\end{proof}

\noindent {\bf Step 2: estimate on $|u_j(\tau,y)-\xi_j(\tau,y)|$.}
\begin{lem}\label{lem:u-xi}
For  all $j \in \{1, \dots, n\}$ and $0\le \d\le 1$, 
\begin{equation}\label{eq:uj-xij}
\|u_j-\xi_j\|_{L^\infty}\le \frac{C_4}{\alpha_0} 
\end{equation}
where $C_4$ is given by \eqref{eq:C4}. 
\end{lem}
\begin{proof}
We recall that $((u_j),(\xi_j))$ is solution of 
\begin{equation}\label{eq:borne1}
  \left\{\begin{array}{l}
      (u_j)_\tau= \a_0 (\xi_j-u_j)\\
      (\xi_j)_\tau\le 2 F_j(\tau,[u(\tau,\cdot)]_{j,m})+\a_0(u_j-\xi_j)+\delta (a_0 +C_0) ((\xi_j)_y)^+
\end{array}\right.
\end{equation}
where we have used \eqref{eq:24}. Using Proposition~\ref{pro:130}, we
deduce that (for $\d\le 1$)
\begin{equation}\label{eq:borne2}
\delta (a_0 +C_0) ((\xi_j)_y)^+\le (a_0 +C_0)(p+2 L_F).
\end{equation}
We now want to bound $F_j(\tau,[u(\tau,\cdot)]_{j,m})$. We have
\begin{align}\label{eq:boundF1}
  F_j(\tau,[u(\tau,\cdot)]_{j,m}(y))=&F_j(\tau,[u(\tau,\cdot)
  -\lfloor u_j(\tau,y) \rfloor]_{j,m}(y))\nonumber\\
  \le & L_F+F_j(\tau,[u(\tau,\cdot)- u_j(\tau,y) ]_{j,m}(y))\nonumber\\
  \le & L_F+L_F \sup_{k\in \{0,\dots ,m\}} (u_{j+k}(\tau,y)-
  u_j(\tau,y)) + \sup_\tau F(\tau, 0,\dots 0)
\end{align}
where we have used the periodicity assumption (A4) for the first line,
the Lipschitz regularity of $F$ for the second and third ones, and the
fact that $u_l$ is non-decreasing with respect to $l$ for the third
line. Moreover for all $i\in\{1,\dots, n\}, \; k\in \{0, \dots m\}$ ,
we have that
\begin{align}\label{eq:boundF2}
  0\le u_{i+k}(\tau,y)-u_i(\tau,y)=&u_{i+k-\left \lceil \frac k
      n\right \rceil n}
  (\tau , y+\left \lceil \frac k n\right \rceil n)-u_i(\tau,y)\nonumber\\
  \le & u_{i}(\tau , y+\left\lceil \frac k n\right\rceil
  n)-u_i(\tau,y)
  \nonumber\\
  \le & 1 + p \left\lceil \frac k n\right\rceil n\nonumber\\
  \le & 1+ p (m+n)
\end{align}
where we have used the periodicity of $u_i$ for the first line, the
monotonicity in $i$ of $u_i$ for the second one and the control of the
oscillation \eqref{eq:sp-osc} for the third one. We then deduce that
$$F_j(\tau,[u_j(\tau,\cdot)]_{j,m}(y))\le  L_F( 2+ p (m+n)) + \sup_\tau F(\tau, 0,\dots 0).$$
Combining this inequality with \eqref{eq:borne1} and \eqref{eq:borne2}, we deduce that  
$$
\left\{\begin{array}{l}
(u_j)_\tau = \a_0 (\xi_j-u_j)\\
 (\xi_j)_\tau\le 2 C_4 +\a_0(u_j-\xi_j)
\end{array}\right.
$$
We now define for all $j\in \Z$ $v_j=\xi_j-u_j$. Classical arguments
from viscosity solution theory show that
$$(v_j)_\tau\le 2(C_4-\a_0 v_j).$$
We then deduce that 
$$v_j\le \frac {C_4} {\a_0}.$$
Using the same arguments with super-solution for $\xi_j$, we get the desired result.
\end{proof}

\noindent {\bf Step 3: control of the time oscillations.}

\noindent
We now explain how to control the time oscillations. The proof is
inspired of \cite{IM}. Let us introduce the following continuous
functions defined for $T>0$
\begin{eqnarray*}
\lambda_+^u(T)=\sup_{j\in \{1,\dots,n\}}\sup_{\tau\ge 0}\frac{u_j(\tau+T,0)-u_j(\tau,0)}{T} \\
  \lambda_- ^u(T)=\inf_{j\in \{1,\dots,n\}}\inf_{\tau\ge 0}\frac{u_j(\tau+T,0)-u_j(\tau,0)}{T}
\end{eqnarray*}
and 
\begin{eqnarray*}
\lambda_+^\xi(T)=\sup_{j\in \{1,\dots,n\}}\sup_{\tau\ge 0}\frac{\xi_j(\tau+T,0)-\xi_j(\tau,0)}{T} \\
\lambda_-^\xi(T)=\inf_{j\in \{1,\dots,n\}}\inf_{\tau\ge  0}\frac{\xi_j(\tau+T,0)-\xi_j(\tau,0)}{T}
\end{eqnarray*}
and 
$$
\lambda_+(T)=\sup(\lambda_+^u(T),\lambda_+^\xi(T))\quad
\mbox{and}\quad \lambda_-(T)=\inf( \lambda_- ^u(T), \lambda_-^\xi(T)).
$$ 
In particular, these functions satisfy $-\infty
\le \lambda_-(T)\le \lambda_+(T) \le + \infty$.\medskip

The goal is to prove that $\l_+(T)$ and $\l_-(T)$ have a common limit
as $T\to \infty$. We would like to apply Lemma \ref{lem:ergo}.

In view of the definition of $\lambda^u_+$ and $\lambda^\xi_+$, we see
that $T \mapsto T \lambda^u_+ (T)$ and $T \mapsto T \lambda^\xi_+ (T)$
are sub-additive. Analogously, $T \mapsto -T \lambda^u_-(T)$ and $T
\mapsto -T \lambda^\xi_-(T)$ are also sub-additive. Hence, if we can
prove that these quantities $\lambda^u_\pm(T), \; \lambda^\xi_\pm(T)$
are finite, we will know that they converge. We will then have to
prove that the limits of $\lambda_+$ and $\lambda_-$ are the same.

\noindent{\bf Step 3.1: first control on the time oscillations}

We first prove that $\lambda_\pm$ are finite.
\begin{lem}
For all $T>0$, 
\begin{equation}\label{eq:boundlambda}
-K_1-\frac {C_1}T\le \lambda_-(T)\le \lambda_+(T)\le K_1+\frac {C_1}T
\end{equation}
where $C_1=\frac{C_4}{\alpha_0}+3+2p$ and $K_1$ is defined in \eqref{eq:C}.
\end{lem}
\begin{proof}
Consider $j \in \{1, \dots, n\}$. 
Using the control of the space oscillations \eqref{eq:sp-osc}, we get that
$$
u_j(\tau,y)\ge \D +py-1\quad
\mbox{and}\quad \xi_j(\tau,y)\ge  \D +py-1
$$
where
$$
\D=\inf_{j\in\{1,\dots,n\}}\inf (u_j(\tau,0), \xi_j(\tau,0)) \, .
$$
Recalling (see Lemma \ref{lem:1}) that $\lfloor\D-p\rfloor+p(y+ \frac
{j}n)-1-K_1t$ is a sub-solution and using the comparison principle on
the time interval $[\tau,\tau+t)$, we deduce that
\begin{equation}\label{eq:time-osc1}
u_j(\tau+t,y)\ge \lfloor\D-p\rfloor+py+ \frac {pj}n-1-K_1t\quad
\mbox{and}\quad \xi_j(\tau+t,y)\ge  \lfloor\D-p\rfloor+py+ \frac {pj}n-1-K_1t \, .
\end{equation}
We now want to estimate $\D$ from below. Let us assume that the
infimum in $\D$ is reached for the index $\bar j \in \{1,\dots, n\}$. 
Then $\bar j \ge j -n$ since $j \in \{1, \dots, n\}$.
We then deduce that
\begin{align*}
p+ \lfloor\D-p\rfloor\ge&\D-1\\
\ge & u_{\bar j}(\tau, 0)-\frac{C_4}{\alpha_0}-1\\
\ge & u_{j-n}(\tau,0)-\frac{C_4}{\alpha_0}-1\\
\ge & u_{j}(\tau,-1)-\frac{C_4}{\alpha_0}-1\\
\ge &u_{j}(\tau,0)-\frac{C_4}{\alpha_0}-2-p
\end{align*}
where we have used \eqref{eq:uj-xij} for the second line, the fact
that $(u_j)_j$ is non-decreasing in $j$ for the third line, the
periodicity of $u_j$ for the fourth line and \eqref{eq:sp-osc} for the
last one.  In the same way, we get that
$$
p+ \lfloor\D-p\rfloor\ge\xi_{j}(\tau,0)-\frac{C_4}{\alpha_0}-2-p.
$$
Injecting this in \eqref{eq:time-osc1}, we get that
\begin{equation}\label{eq:time-osc3}
u_j(\tau+t,y)\ge u_{j}(\tau,0)-C_1 +py -K_1t
\end{equation}
and
$$
\xi_j(\tau+t,y)\ge  \xi_{j}(\tau,0)-C_1+py-K_1t.
$$
In the same way, we also get
\begin{equation}\label{eq:time-osc4}
u_j(\tau+t,y)\le u_{j}(\tau,0)+ C_1 +py+K_1t 
\end{equation}
and
$$
\xi_j(\tau+t,y)\le  \xi_{j}(\tau,0)+C_1+py+K_1t.
$$
Taking $y=0$, we finally get \eqref{eq:boundlambda}.
\end{proof}

\noindent{\bf Step 3.2: Refined control on the time oscillations}\\
We now estimate $\lambda_+ - \lambda_-$ in order to prove that they
have the same limit.
\begin{lem}\label{lem:lambda+-lambda-}
For all $T>0$,
$$
|\lambda_+(T)-\lambda_-(T)|\le \frac {C_2}T 
$$
where $C_2=6+\frac{4C_4}{\alpha_0}+3p+2C_1+2K_1$.
\end{lem}
\begin{proof}
By definition of $\lambda_\pm(T)$, for all $\e>0$, there exists
$\tau^\pm\ge 0$ and $v^\pm \in \{u_1,\dots u_n, \xi_1,\dots \xi_n\}$
such that
$$
\left|\lambda_\pm(T)-\frac{v^\pm(\tau^\pm+T,0)-v^\pm(\tau^\pm,0)}{T}\right|\le \e.
$$
Consider $j \in \{1,\dots, n\}$. 
We choose $\beta\in [0,1)$ such that $\tau^+-\tau^--\beta=k\in \Z$ and we set
$$
\Delta^u_j=u_j(\tau^+,0)-u_j(\tau^-+\beta,0), \quad 
\Delta^\xi_j=\xi_j(\tau^+,0)-\xi_j(\tau^-+\beta,0)
$$ 
and
$$
\Delta=\sup_{j\in\{1,\dots, n\}}\sup(\Delta^u_j,\Delta^\xi_j).
$$
Using \eqref{eq:sp-osc}, we get that
$$
u_j(\tau^+,y)\le u_j(\tau^-+\beta,y)+2+\lceil\Delta\rceil
\quad{\rm and}\quad \xi_j(\tau^+,y)\le \xi_j(\tau^-+\beta,y)+2+\lceil\Delta\rceil \, .
$$
Using the comparison principle, we then deduce that
\begin{equation}\label{eq:time-osc2}
u_j(\tau^++T,y)\le u_j(\tau^-+\beta+T,y)+2+\lceil\Delta\rceil
\quad{\rm and}\quad \xi_j(\tau^++T,y)\le \xi_j(\tau^-+\beta+T,y)+2+\lceil\Delta\rceil.
\end{equation}
We now want to estimate $\lceil\Delta\rceil$ from above. Let us assume
that the maximum in $\Delta$ is reached for the index $\bar j$. We
then have for all $j\in \{1,\dots, n\}$
\begin{align*}
\lceil\Delta\rceil\le& u_{\bar j}(\tau^+,0)-u_{\bar j}(\tau^-+\beta,0)+\frac{2C_4}{\alpha_0}+1\\
\le&u_{j+n}(\tau^+,0)-u_{j-n}(\tau^-+\beta,0)+\frac{2C_4}{\alpha_0}+1\\
\le&u_{j}(\tau^+,1)-u_{j}(\tau^-+\beta,-1)+\frac{2C_4}{\alpha_0}+1\\
\le&u_{j}(\tau^+,0)-u_{j}(\tau^-+\beta,0)+\frac{2C_4}{\alpha_0}+3+2p
\end{align*}
where we have used \eqref{eq:uj-xij} for the first line, the fact that
$(u_j)_j$ is non-decreasing in $j$ for the second line, the
periodicity of $u_j$ for the third line and \eqref{eq:sp-osc} for the
last one. In the same way, we also get
$$
\lceil\Delta\rceil\le\xi_{j}(\tau^+,0)-\xi_{j}(\tau^-+\beta,0)+\frac{2C_4}{\alpha_0}+3+2p \, .
$$
Injecting this in \eqref{eq:time-osc2}, we get
$$u_j(\tau^++T,y)\le u_j(\tau^-+\beta+T,y)+5+\frac{2C_4}{\alpha_0}+2p+\Delta_j^u$$
and 
$$
\xi_j(\tau^++T,y)\le \xi_j(\tau^-+\beta+T,y)+5+\frac{2C_4}{\alpha_0}+2p+\Delta_j^\xi \, .
$$
Taking $y=0$ and using \eqref{eq:time-osc3} (with $\tau=\tau^-$ and
$t=\beta $) and \eqref{eq:time-osc4} (with $\tau=\tau^-+T$ and
$t=\beta$), we get
$$
u_j(\tau^++T,0)-u_j(\tau^+,0)\le u_j(\tau^-+T,0)-u_j(\tau^-,0)+5+\frac{2C_4}{\alpha_0}+2p+2C_1+2K_1 \, .
$$
In the same way, we get
$$
\xi_j(\tau^++T,0)-\xi_j(\tau^+,0)\le
\xi_j(\tau^-+T,0)-\xi_j(\tau^-,0)+5+\frac{2C_4}{\alpha_0}+2p+2C_1+2K_1 \, .
$$
Using also \eqref{eq:uj-xij}, \eqref{eq:sp-osc} and the fact that
$(u_j)_j$ and $(\xi_j)_j$ are non-decreasing in $j$, we finally get
$$
v^+(\tau^++T,0)-v^+(\tau^+,0)\le v^-(\tau^-+T,0)-v^-(\tau^-,0)+C_2 \, .
$$
The comparison of $u_j$ and $\xi_j$ makes appear the additional constant $2C_4/\alpha_0$, 
and the comparison between $u_j$ and $u_k$ (and similarly between
$\xi_j$ and $\xi_k$) creates an additional constant $1+p$.
Indeed, we have
$$
u_j (\tau,0)-u_k(\tau,0) = u_{j+n} (\tau,1) -u_k(\tau,0)\le u_{j+n}
(\tau,0) -u_k(\tau,0)+1+p\le 1+p .
$$ 
This explains the value of the new constant $C_2$.

This implies that
$$
T\lambda_+(T)\le T\lambda_-(T)+2\e+C_2 \, .
$$
Since this is true for all $\e>0$, the proof of the lemma is complete. 
\end{proof}

\noindent{\bf Step 3.3: Conclusion}\\
We now can conclude that
$\lim_{T\to +\infty } \lambda_\pm (T)$ are equal. If $\lambda$ denotes the common limit, we also have, by Lemma \ref{lem:ergo}, that for every $T>0$,
$$\lambda_-(T)\le \lambda \le \lambda_+(T).$$

Moreover, by Lemma \ref{lem:lambda+-lambda-}, we have
$$\lambda_+(T)\le \lambda_-(T)+\frac {C_2} {T}$$
and so
$$\lambda_-(T)\le \lambda \le\lambda_-(T)+\frac {C_2} {T}$$
We finally deduce (using a similar argument for $\lambda_+$) that
$$
|\lambda_\pm (T) -\lambda | \le \frac{C_2}T.
$$
Combining this estimate and \eqref{eq:sp-osc}, we get with $T=\tau$
$$
|u_j(\tau,y)-u_j(0,0)-py-\lambda \tau|\le C_2+1
$$
and 
$$
|\xi_j(\tau,y)-\xi_j(0,0)-py-\lambda \tau|\le C_2+1 \, .
$$
This  finally implies \eqref{eq:41} with $C_3=C_2+1$. 
 \end{proof}

 \subsection{Construction of hull functions for general Hamiltonians}
\label{subsec:hull}

In this subsection, we construct hull functions for a general
Hamiltonian $G_j$. As we shall see, this is a straightforward
consequence of the construction of time-space periodic solutions of
\eqref{eq:121}; see Proposition~\ref{pro:122} and
Corollary~\ref{pro:122bis} below.  We will then prove that the time
slope obtained in Proposition~\ref{pro:11} is unique and that the map
$p \mapsto \lambda$ is continuous; see Proposition~\ref{pro:129}
below.  \medskip

Given $p >0$, we consider the equation in $\R \times \R$
\begin{equation}\label{eq:121}
\left\{\begin{array}{l}
\left\{\begin{array}{l}
(u_j)_\tau=\alpha_0 (\xi_j-u_j)\\
(\xi_j)_\tau=
G_j(\tau,[u(\tau,\cdot)]_{j,m},\xi_j,\inf_{y'\in\R}\left(\xi_j(\tau,y')-py'\right)+py-\xi_j(\tau,y),(\xi_j)_y)\\
\end{array}\right.    
\\ \\
\left\{\begin{array}{l}
u_{j+n}(\tau,y)=u_j(\tau,y+1)\\
\xi_{j+n}(\tau,y)=\xi_j(\tau,y+1) \, ,
\end{array}\right.    
\end{array}\right.    
\end{equation}
where $G_j=G_j^\d$ is given in \eqref{defi:gdelta1} for $\d\ge 0$.
Then we have the following result
\begin{pro}\label{pro:122}{\bf (Existence of time-space periodic
    solutions of \eqref{eq:121})}\\
  Let $0\le \delta\le 1$, $a_0\in \R$ and $p>0$. Assume {\rm
    (A1)-(A6)}. Then there exist functions $((u_j^\infty)_j,
  (\xi_j^\infty)_j)$ solving \eqref{eq:121} on $\R\times \R$ and a real
  number $\lambda \in \R$ satisfying for all $\tau, y \in \R$, $j\in
  \{1,\dots,n\}$
\begin{eqnarray}
\label{eq:125}
|u_j^\infty(\tau,y)-py-\lambda \tau|&\le& 2\lceil C_3\rceil   \\
\nonumber |\xi_j^\infty(\tau,y)-py-\lambda \tau|&\le& 2\lceil C_3\rceil \, .  
\end{eqnarray}
Moreover $((u_j^\infty)_j,(\xi_j^\infty)_j)$ satisfies for $j\in \{1,\dots, n\}$
\begin{equation}\label{eq:123}
\left\{\begin{array}{l}
u_j^\infty(\tau,y+1/p)=u_j^\infty(\tau,y)+1\\
u_j^\infty(\tau+1,y)=u_j^\infty(\tau,y+\lambda/p)\\
(u_j^\infty)_y(\tau,y)\ge 0\\
u_{j+1}(\tau,y)\ge u_j(\tau,y)
\, .
\end{array}\right.
\quad 
\left\{\begin{array}{l}
\xi_j^\infty(\tau,y+1/p)=\xi_j^\infty(\tau,y)+1\\
\xi_j^\infty(\tau+1,y)=\xi_j^\infty(\tau,y+\lambda/p)\\
(\xi_j^\infty)_y(\tau,y)\ge 0\\
\xi_{j+1}(\tau,y)\ge \xi_j(\tau,y)
\, .
\end{array}\right.
\end{equation}
Eventually, when the Hamiltonians $G_j$ are independent on $\tau$, we can choose $u_j^\infty$ and $\xi_j^\infty$
independent on $\tau$.
\end{pro}
By considering for all $\tau,z \in \R$ 
\begin{equation}\label{eq:h}
\left\{\begin{array}{l}
h_j(\tau,z)=u_j^\infty(\tau,(z-\lambda\tau)/p)\quad {\rm if}\; j\in\{1,\dots,n\}\\
h_{j+n}(\tau,z)=h_j(\tau,z+p)\quad {\rm otherwise}
\end{array}\right.
\end{equation}
and for all $\tau,z \in \R$, 
\begin{equation}\label{eq:g}\left\{\begin{array}{l}
g_j(\tau,z)=\xi_j^\infty(\tau,(z-\lambda\tau)/p)\quad {\rm if}\; j\in\{1,\dots,n\}\\
g_{j+n}(\tau,z)=g_j(\tau,z+p)\quad {\rm otherwise}
\end{array}\right.
\end{equation}
 we immediately get the  following corollary
\begin{cor} \label{pro:122bis} {\bf (Existence of hull functions)}\\
Assume {\rm (A1)-(A6)}. There exists a hull function $((h_j)_j, (g_j)_j)$ in the sense
of Definition~\ref{defi:1n} satisfying 
$$
| h_j(\tau,z) - z |\le  2 \lceil C_3\rceil
$$
and 
$$
| g_j(\tau,z) - z |\le  2 \lceil C_3\rceil
$$
\end{cor}
We now turn to the proof of Proposition~\ref{pro:122}. 
\begin{proof}[Proof of Proposition \ref{pro:122}]
 
The proof is performed in three steps. In the first one, we construct
sub- and super-solutions of \eqref{eq:121} in $\R \times \R$ with good
translation invariance properties (see the first two lines of \eqref{eq:123}). 
We next apply Perron's method in order to get a (possibly discontinuous) solution
satisfying the same properties. Finally, in Step~3, we prove that if the functions $G_j$
do not depend on $\tau$, then we can construct a solution in such a way that
it does not depend on $\tau$ either. 

\medskip

\noindent {\bf Step 1: global sub- and super-solution}

\noindent By Proposition~\ref{pro:11}, we know that the solution
$(u_j, \xi_j)$ of \eqref{eq:22n}, \eqref{eq:22n-ic} with initial data
$u_0(y)=py=\xi_0(y)$ satisfies on $[0,+\infty )\times \R$
\begin{equation}\label{eq:124}
\left\{ \begin{array}{l}
(u_j)_y\ge 0,  \\
|u_j(\tau,y)-py-\lambda \tau|\le C_3, \\
|u_j(\tau,y+y')-u_j(\tau,y)-py'|\le 1,\\
u_{j+1}(\tau,y)\ge u_j(\tau,y) ,
\end{array}\right.
\quad
\left\{ \begin{array}{l}
(\xi_j)_y\ge 0,  \\
|\xi_j(\tau,y)-py-\lambda \tau|\le C_3, \\
|\xi_j(\tau,y+y')-\xi_j(\tau,y)-py'|\le 1,\\
\xi_{j+1}(\tau,y)\ge \xi_j(\tau,y) \, .
\end{array}\right.
\end{equation}

We first construct a sub-solution and a super-solution of \eqref{eq:121} for $\tau \in \R$ (and not
only $\tau \ge 0$) that also satisfy the first two lines of \eqref{eq:123}, \textit{i.e.} satisfy
for all $k,l \in \Z$,
\begin{equation}\label{eq:transinv}
U (\tau + k, y) = U (\tau,y + \lambda \frac{k}p) \quad \text{ and } \quad
U (\tau,y+ \frac{l}p) = U (\tau,y) + l\, .
\end{equation}
To do so, we consider for $j\in \{1,\dots,n\}$ two sequences of
functions (indexed by $m \in \N$,  \; $m\to \infty$)
$$
u_j^m(\tau,y)=u_j(\tau+m,y)-\lfloor \lambda m \rfloor, \quad \xi_j^m(\tau,y)=\xi_j(\tau+m,y)-\lfloor\lambda m\rfloor
$$
and consider 
$$
\overline{u}_j={\limsup_{m\to +\infty}}^* u_j^m, \quad \overline{\xi}_j={\limsup_{m\to +\infty}}^* \xi_j^m
$$
$$
\underline{u}_j={\liminf_{m\to +\infty}}_* u_j^m, \quad \underline{\xi}_j={\liminf_{m\to +\infty}}_* \xi_j^m \, .
$$
We first remark  that thanks to \eqref{eq:41}, all these semi-limits are finite. 
We also remark that for all $k,l \in \Z$, 
$$
(\overline{u}_j (\tau+k,y - k \lambda /p + l/p)-l,\overline{\xi}_j (\tau+k,y - k \lambda /p + l/p)-l )
$$ 
is a sub-solution of \eqref{eq:121}. A similar remark can be done for the super-solutions $(\underline u_j, \underline \xi_j)_j$. 

Now a way to construct sub-solution (resp. a super-solution) of \eqref{eq:22n} satisfying \eqref{eq:transinv}
is to consider 
\begin{equation}\label{eq:126}
\left\{\begin{array}{l}
\overline{u}_j^\infty(\tau,y)=\left( \sup_{k,l\in\Z}
\left(\overline{u}_j(\tau+k,y-k\lambda/p+l/p)-l\right) \right)^*,
\smallskip \\
\overline{\xi}_j^\infty(\tau,y)=\left( \sup_{k,l\in\Z}
\left(\overline{\xi}_j(\tau+k,y-k\lambda/p+l/p)-l\right) \right)^*,
\end{array}\right.
\end{equation}
and
\begin{equation}\label{eq:127}
 \left\{\begin{array}{l}
\underline{u}_j^\infty(\tau,y)=\left( \inf_{k,l\in\Z}
\left(\underline{u}_j(\tau+k,y-k\lambda/p+l/p)-l\right) \right)_* ,
\smallskip \\
\underline{\xi}_j^\infty(\tau,y)=\left( \inf_{k,l\in\Z}
\left(\underline{\xi}_j(\tau+k,y-k\lambda/p+l/p)-l\right) \right)_* .
\end{array}\right.
\end{equation}
Notice that $\overline{u}_j^\infty$, $\underline{u}_j^\infty,\;
\overline \xi_j^\infty$ and $\underline \xi_j^\infty$ satisfy moreover
\eqref{eq:124} on $\R\times \R$.  Therefore we have in particular
$$
\overline{u}_j^\infty \le \underline{u}_j^\infty + 2\lceil C_3 \rceil
\quad {\rm and}\quad \overline{\xi}_j^\infty \le
\underline{\xi}_j^\infty + 2\lceil C_3 \rceil \, .
$$
\medskip

\noindent {\bf Step 2: existence by Perron's method}

\noindent 
Applying Perron's method we see that the lowest-$*$ super-solution
$((u_j^\infty)_j, (\xi_j^\infty)_j)$ lying above
$((\overline{u}_j^\infty)_j, (\overline \xi_j^\infty)_j)$ is a
(possibly discontinuous) solution of \eqref{eq:121} on $\R\times \R$
and satisfies
$$
\overline{u}_j^\infty\le u_j^\infty \le \underline{u}_j^\infty + 2\lceil C_3 \rceil \quad {\rm and}\quad 
\overline{\xi}_j^\infty\le \xi_j^\infty \le \underline{\xi}_j^\infty + 2\lceil C_3 \rceil\, .
$$
We next prove that $u^\infty$ satisfies \eqref{eq:123}. 
 For $j\in \{1,\dots,n\}$, let us consider
\begin{equation}\label{eq:128}
\tilde{u}_j^\infty(\tau,y)=\left(\inf_{k,l\in\Z}
\left({u}_j^\infty (\tau+k,y-k\lambda/p+l/p)-l\right)\right)_*
\end{equation}
$$\tilde{\xi}_j^\infty(\tau,y)=\left(\inf_{k,l\in\Z}
\left({\xi}_j^\infty (\tau+k,y-k\lambda/p+l/p)-l\right)\right)_*
$$
By construction the family $((\tilde{u}_j^\infty)_j, (\tilde
\xi_j^\infty)_j)$ is a super-solution of \eqref{eq:121} and is again
above the sub-solution $((\overline{u}_j^\infty)_j, (\overline
\xi_j^\infty)_j)$. Therefore from the definition of
$((u^\infty_j)_j,(\xi^\infty_j)_j)$, we deduce that
$$\tilde{u}_j^\infty = u_j^\infty\quad {\rm and}\quad \tilde{\xi}_j^\infty = \xi_j^\infty$$
which implies that $u_j^\infty$ and $\xi_j^\infty$ satisfy
\eqref{eq:transinv}, \textit{i.e} the first two equalities of
\eqref{eq:123}.

Similarly, we can consider, for $j\in \{1,\dots,n\}$
$$\hat{u}_j^\infty(\tau,y)=\left(\inf_{b\in [0,+\infty)} u_j^\infty(\tau,y+b)\right)_*$$
$$\hat{\xi}_j^\infty(\tau,y)=\left(\inf_{b\in [0,+\infty)} \xi_j^\infty(\tau,y+b)\right)_*$$

which is again super-solution above the sub-solution
$((\overline{u}^\infty_j)_j,(\overline \xi_j^\infty)_j)$. Therefore
$$\hat{u}^\infty_j = u^\infty_j\quad {\rm and}\quad \hat{\xi}^\infty_j = \xi^\infty_j$$
which implies that $u_j^\infty$ and $\xi_j^\infty$ are  non-decreasing in $y$, \textit{i.e.} 
the third line of \eqref{eq:123} is satisfied.

Let us now prove that $u_j^\infty$ and $\xi_j^\infty$ are  non-decreasing in $j$. We consider, for $j\in \{1,\dots,n\}$
$$\check{u}_j^\infty(\tau,y)=\left(\inf_{ k\ge 0} u_{j+k}^\infty(\tau,y)\right)_*=\left(\inf_{0\le k< n} u_{j+k}^\infty(\tau,y)\right)_*$$
$$\check{\xi}_j^\infty(\tau,y)=\left(\inf_{k\ge 0} \xi_{j+k}^\infty(\tau,y)\right)_*=\left(\inf_{0\le k< n} \xi_{j+k}^\infty(\tau,y)\right)_*.$$
The fact that this is a super-solution uses assumption (A6). Indeed, let us
assume that the infimum for $u_j$ is reached for the index $k_u$ and that
the infimum for $\xi_j$ is reached for the index $k_\xi$. Then, formally,
on one hand we have
\begin{align*}
(\check u_j^\infty)_\tau(\tau,y)=& \a_0(\xi_{j+k_u}^\infty(\tau,y)-u_{j+k_u}^\infty(\tau,y))\\
\ge& \a_0(\xi_{j+k_\xi}^\infty(\tau,y) - u_{j+k_u}^\infty(\tau,y))\\
\ge& \a_0 (\check \xi_j^\infty(\tau,y)- \check u_j^\infty(\tau,y))
\end{align*}
where we have used the fact that $\xi_{j+k_u}^\infty(\tau,y)\ge \xi_{j+k_\xi}^\infty(\tau,y)$.
On the other hand, we have
\begin{align*}
(\check \xi_j^\infty)_\tau (\tau, y)=& G_{j+k_\xi}(\tau, [u^\infty(\tau,\cdot)]_{j+k_\xi}(y), \xi_{j+k_\xi}^\infty(\tau,y), \inf_{y'}(\xi_{j+k_\xi}^\infty(\tau, y')-py')+py -\xi_{j+k_\xi}^\infty(\tau,y), (\xi_{j+k_\xi}^\infty)_y)\\
\ge &
G_{j+k_\xi}(\tau, [\check u^\infty(\tau,\cdot)]_{j+k_\xi}(y), \check \xi_{j}^\infty(\tau,y), \inf_{y'}(\check \xi_{j}^\infty(\tau, y')-py')+py -\check \xi_{j}^\infty(\tau,y), (\check \xi_{j}^\infty)_y)\\
\ge &
G_{j+k_\xi-1}(\tau, [\check u^\infty(\tau,\cdot)]_{j+k_\xi-1}(y), \check \xi_{j}^\infty(\tau,y), \inf_{y'}(\check \xi_{j}^\infty(\tau, y')-py')+py -\check \xi_{j}^\infty(\tau,y), (\check \xi_{j}^\infty)_y)\\
\ge & \dots\\
\ge &
G_{j}(\tau, [\check u^\infty(\tau,\cdot)]_{j}(y), \check \xi_{j}^\infty(\tau,y), \inf_{y'}(\check \xi_{j}^\infty(\tau, y')-py')+py -\check \xi_{j}^\infty(\tau,y), (\check \xi_{j}^\infty)_y)
\end{align*}
where we have used the fact that $u^\infty_{j+k_\xi+k}\ge \check u^\infty_{j+k_\xi+k}$ and  $\xi_{j+k_\xi}^\infty(\tau, y')\ge\check \xi_{j}^\infty(\tau, y')$ joint to  the monotonicity assumption of $G$ in the variable $V_i$ and $a$ for the first inequality and assumtion (A6) joint to the fact that $\check u^\infty_j$ is non-decreasing in $j$ (by construction) for the other inequalities.

We then conclude that $(\check u_j^\infty, \check \xi_j^\infty)$
is again super-solution above the sub-solution
$((\overline{u}^\infty_j)_j,(\overline \xi_j^\infty)_j)$. Therefore
$$u^\infty_j=\check u^\infty_j\quad {\rm and}\quad  \xi^\infty_j=\check\xi_j^\infty$$
which implies that $u_j^\infty$ and $\xi_j^\infty$ are  non-decreasing in $j$, \textit{i.e.} 
the forth line of \eqref{eq:123} is satisfied.

Finally, the function $((u_j^\infty-\lceil C_3\rceil)_j, (\xi_j^\infty-\lceil C_3\rceil)_j) $ still satisfies 
\eqref{eq:123} and also satisfies  \eqref{eq:125}.
\medskip

\noindent {\bf Step 3: Further properties when the  $G_j$ are independent on $\tau$}

\noindent
When the $G_j$ do not depend on $\tau$, we can apply Steps 1 and 2 with $k\in\Z$
 in \eqref{eq:126}, \eqref{eq:127} and \eqref{eq:128} replaced with $k\in\R$. This
implies  that the hull function $((h_j)_j,(g_j)_j)$ does not depend  on $\tau$.
This ends the proof of the proposition.
\end{proof}
\medskip

\begin{pro}[\bf Definition and continuity of the effective Hamiltonian] \label{pro:129}
~

Consider $p>0$ and assume {\rm (A1)-(A6)}. Then
\begin{itemize}
\item
there exists a unique real number
$\lambda \in \R$  such that there exists a  solution $((u_j^\infty)_j,(\xi_j^\infty)_j)$
of  \eqref{eq:121} on $\R\times \R$ such that there exists $C>0$ such that for all $\tau$, 
\begin{equation}\label{eq:125bis}
|h_j(\tau,z)-z|\le  C \quad {\rm and} \quad   |g_j(\tau,z)-z|\le C,
\end{equation}
where the $h_j$ and the $g_j$ are defined in \eqref{eq:h} and
\eqref{eq:g}; moreover, we can choose $C= 2\lceil C_3\rceil$ with
$C_3$ given in \eqref{eq:C4};
\item
if $\lambda$ is seen as a function $\overline{G}$ of $p$ ($\lambda=
\overline{G}(p)$), then this
function $\overline{G} : (0,+\infty) \to \R$ is continuous.
\end{itemize}
\end{pro}
Before to prove this proposition, let us give the proof of Theorem \ref{th:2}. 
\begin{proof}[Proof of Theorem \ref{th:2}]
Just apply Proposition  \ref{pro:129} with $G=F$.
\end{proof}

\begin{proof}[Proof of Proposition \ref{pro:129}]
The proof follows classical arguments. However, we give it for the reader's convenience. 
The proof is divided in two steps. 

\noindent {\bf Step 1: Uniqueness of $\lambda$}

\noindent
Given some  $p\in (0,+\infty )$, assume that there exist
$\lambda_1,\lambda_2\in\R$ with their corresponding hull functions
$((h_j^1)_j,(g_j^1)_j),((h_j^2)_j,(g_j^2)_j)$. Then define for $ i=1,2$, $j\in\{1,\dots,n\}$
$$
u_j^i(\tau,y)=h_j^i(\tau, \lambda_i\tau +py)\quad {\rm and}\quad \xi_j^i(\tau,y)=g_j^i(\tau, \lambda_i\tau +py)
$$
which are both solutions of equation (\ref{eq:22n}) on
$[0,+\infty)\times \R$. 
By Corollary \ref{pro:122bis}, we know that $h_j$ and $g_j$ satisfy
\eqref{eq:125bis}.  Then we have with $C= 2\lceil C_3\rceil$
$$u_j^1(0,y) \le u_j^2(0,y) + 2C\quad {\rm and}\quad \xi_j^1(0,y) \le \xi_j^2(0,y) + 2C $$
which implies (from the comparison principle) for all $(\tau,y)\times
[0,+\infty )\times \R$
$$
u_j^1(\tau,y) \le u_j^2(\tau,y) + 2C \quad {\rm and}\quad \xi_j^1(\tau,y) \le \xi_j^2(\tau,y) + 2C  \, .
$$
Using the fact that $h_j^i(\tau+1,z)=h_j^i(\tau,z)$ and
$g_j^i(\tau+1,z)=g_j^i(\tau,z)$, we deduce that for $\tau=k\in\N$ and
$y=0$ we have
$$
h_j^1(0, \lambda_1k) \le h_j^2(0, \lambda_2k) +2C \quad {\rm and}\quad 
g_j^1(0, \lambda_1k) \le g_j^2(0, \lambda_2k) +2C
$$
which implies by (\ref{eq:125bis})
$$
\lambda_1k \le \lambda_2k +4 C \, .
$$
Because this is true for any $k\in\N$, we deduce that
$$
\lambda_1 \le \lambda_2 \, .
$$
The reverse inequality is obtained exchanging $((h^1_j)_j,(g_j^1)_j)$
and $((h^2_j)_j,(g_j^2)_j)$.  We finally deduce that
$\lambda_1=\lambda_2$, which proves the uniqueness of the real
$\lambda$, that we call $\overline{G}(p)$.
\medskip

\noindent {\bf Step 2: Continuity of the map $p\mapsto \overline{G}(p)$}

\noindent
Let us consider a sequence $(p_m)_m$ such that $p_m\to p>0$.
Let $\lambda_m=\overline{G}(p_m)$ and $((h_j^m)_j,(g_j^m)_j)$ be the corresponding hull
functions. From Corollary~\ref{pro:122bis}, we can choose these hull
functions such that for $j\in \{1,\dots,n\}$
$$
|h_j^m(\tau,z)-z|\le  2\lceil C_3\rceil, \quad {\rm and}\quad 
|g_j^m(\tau,z)-z|\le  2\lceil C_3\rceil
$$
and we have
$$
|\lambda_m|\le C_4
$$
where we recall that $C_4$ is defined in (\ref{eq:C4}). Remark that
both $C_3$ and $C_4$ depends on $p_m$, but can be bounded for $p_m$ in
a neighbourhood of $p$.  We deduce in particular that there exists a
constant $C_5>0$ such that
$$
|h_j^m(\tau,z)-z|\le C_5,\quad |g_j^m(\tau,z)-z|\le C_5 \quad \mbox{and}\quad |\lambda_m|\le C_5\, .
$$
Let us consider  a limit $\lambda_\infty$ of $(\lambda_m)_m$, and let us define
$$
\overline{h}_j=\limsup_{m\to +\infty}{}^* h_j^m,\quad {\rm and}\quad  
\overline{g}_j=\limsup_{m\to +\infty}{}^* g_j^m\, .
$$
This family of functions $((\overline{h}_j)_j,(\overline{g}_j)_j)$ is such that the family
$$
((\overline{u}_j(\tau,y))_j,(\overline{\xi}_j(\tau,y))_j)=
((\overline{h}_j(\tau,\lambda_\infty \tau+py))_j,(\overline{g}_j(\tau,\lambda_\infty \tau+py))_j)
$$
is a sub-solution of \eqref{eq:121} on $\R\times \R$.
On the other hand, if  $((h_j)_j,(g_j)_j)$ denotes the hull function associated with $p$ and
$\lambda=\overline{G}(p)$,
then 
$$
((u_j(\tau,y))_j,(\xi_j(\tau,y))_j)=(({h}_j(\tau,\lambda \tau+py))_j,({g}_j(\tau,\lambda \tau+py))_j)
$$
is a solution of (\ref{eq:121}) on $\R\times \R$.
Finally, as in Step 1, we conclude that
$$
\lambda_\infty \le \lambda\, .
$$
Similarly, considering
$$
\underline{h}_j=\liminf_{m\to +\infty}{}_* h_j^m\quad {\rm and}
\quad \underline{g}_j=\liminf_{m\to +\infty}{}_* g_j^m
$$
we can show that
$$
\lambda_\infty \ge \lambda \, .
$$
Therefore $\lambda_\infty =\lambda$ and this proves that
$\overline{G}(p_m) \to \overline{G}(p)$; the continuity of the
map $p\mapsto \overline{G}(p)$ follows and
this ends the proof of the proposition.
\end{proof}

\section{Construction of Lipschitz continuous approximate hull functions}
\label{s5}

\setcounter{equation}{0}

When proving the Convergence Theorem \ref{th:3n}, we explained that, on the
one hand, it is necessary to deal with hull functions
$(h,g)=((h_j(\tau,z))_j, (g_j (\tau,z))_j)$ that are uniformly continuous
in $z$ (uniformly in $\tau$ and $j$) in order to apply Evans' perturbed
test function method; on the other hand, given some $p>0$, we also know
some Hamiltonians $F_j$, with corresponding effective Hamiltonian
$\overline{F}(p)$, such that every corresponding hull function $h_j$ is
necessarily discontinuous in $z$ for $\alpha_0=+\infty$ (see
\cite{A1,FIM2}).  Recall that a hull function $(h,g)$ solves in particular
\begin{equation}\label{eq:hull}
\left\{\begin{array}{l}
(h_j)_\tau + \lambda (h_j)_z =  \alpha_0 (g_j -h_j)\\
(g_j)_\tau + \lambda (g_j)_z = 2 F_j(\tau, [h(\tau,\cdot)]_{j,m})+ \alpha_0 (h_j -g_j)\\
\end{array}\right.
\end{equation}
 with $\lambda=\overline{F}(p)$
and
$$
h_{j+n}(\tau,z)=h_j(\tau,z+p), \; g_{j+n}(\tau,z)=g_j(\tau,z+p) \, .
$$
We overcome this difficulty as in \cite{FIM2} (see also \cite{fim,IM,imr}). 

We build approximate Hamiltonians $G^\delta$ with corresponding effective
Hamiltonians $\lambda^\delta=\overline{G}^\delta(p)$, and
corresponding hull functions $(h^\delta,g^\delta)$, such that
$$
\left\{\begin{array}{l}
(h_j ^\delta,g_j^\delta) \quad \mbox{is Lipschitz continuous with respect to } z \text{ uniformly
  in } \tau \text{ and } j\\
\overline{G}^\delta(p) \to \overline{F}(p)  \quad \mbox{as}\quad  \delta\to
0\\
(h^\d,g^\d) \mbox{ is a sub-/super-solution of \eqref{eq:hull}.} 
\end{array}\right.
$$
We will show that it is enough to  choose for $\d\ge 0$
\begin{equation}\label{defi:gdelta}
G_j^\delta(\tau,V,r,a,q)=2F_j(\tau,V)+\alpha_0 (V_0 -r) +\delta(a_0+a)q^+ 
\end{equation}
with $a_0 \in \R$ (in fact, we will consider $a_0 = \pm 1$).
\medskip

We have the following variant of Corollary \ref{pro:122bis}.
\begin{pro}[\bf Existence of Lipschitz continuous approximate hull functions]\label{pro:135}
~ \\
Assume {\rm (A1)-(A3)}. Given $p>0$, $0<\delta \le 1$ and $a_0\in\R$, then there
exists a family of Lipschitz continuous functions $((h_j)_j,(g_j)_j)$ satisfying for $j\in\{1,\dots,n\}$ 
\begin{equation}\label{eq:136}
\left\{\begin{array}{l}
h_j(\tau,z+1)=h_j(\tau,z)+1\\
h_j(\tau+1,z)=h_j(\tau,z) \\
\\
0\le (h_j)_z \le 1 + \frac{2L_F}{p\delta}
\end{array}\right.
\quad
\left\{\begin{array}{l}
g_j(\tau,z+1)=g_j(\tau,z)+1\\
g_j(\tau+1,z)=g_j(\tau,z) \\
\\
0\le (g_j)_z \le 1 + \frac{2L_F}{p\delta}
\end{array}\right.
\end{equation}
and there exists $\lambda\in\R$ such that
\begin{equation}\label{eq:141}
\left\{\begin{array}{l}
\left\{\begin{array}{rl}
(h_j)_\tau +\lambda (h_j)_z =& \alpha_0 (g_j - h_j) \\
(g_j)_\tau + \lambda (g_j)_z =& 2 F_j(\tau, [h(\tau,\cdot)]_{j,m}) + \alpha_0 (h_j -g_j) \\
& + \delta p \left\{a_0
+\inf_{z'\in\R}\left(h_j(\tau,z')-z') +z-h_j(\tau,z)\right)\right\}(h_j)_z   \\
\end{array}\right.
\\ \\
\left\{\begin{array}{rl}
h_{j+n}(\tau,z)=&h_j(\tau,z+p) \\
g_{j+n}(\tau,z) =&g_j (\tau,z+p) 
\end{array}\right.
\end{array}\right.
\end{equation}
and for all $\tau,z,z' \in \R$
\begin{equation}\label{eq:147bis}
\left| h_j(\tau,z')-z' +z-h_j(\tau,z)\right|\le 1 \quad \text{and} \quad
\left| g_j(\tau,z')-z' +z-g_j(\tau,z)\right|\le 1 
\, .
\end{equation}
Moreover there exists a constant $C_4>0$ defined in \eqref{eq:C4} such that 
\begin{equation}\label{eq:140}
|\lambda|\le C_4   
\end{equation}
and for all $(\tau,z)\in \R\times \R$, 
\begin{equation}\label{eq:138}
|h(\tau,z)-z|\le C(C_4,p,\alpha_0, \delta |a_0|p) 
  \, ,
\end{equation}
$$|g(\tau,z)-z|\le C(C_4,p,\alpha_0, \delta |a_0|p) 
  \, .$$
Moreover, when the $F_j$ do not depend on $\tau$, we can choose the hull
function $((h_j)_j,(g_j)_j)$ such that it does not depend on $\tau$ either.
\end{pro}
\begin{proof}[Proof of Proposition \ref{pro:135}]

\noindent
The construction follows the one made in Proposition~\ref{pro:11} and
Proposition~\ref{pro:122}. However, Proposition \ref{pro:122} has to be
adapted. Indeed, since we want to construct a Lipschitz continuous function
with a precise Lipschitz estimate, we do not want to use Perron's method. 
This is the reason why here we can use a space-time Lipschitz
estimate of $((u_j),(\xi_j))$  to get enough compacity to pass to the
limit. 

The space Lipschitz estimate comes from Proposition \ref{pro:130}. The
time Lipschitz estimate of the $u_j$'s follows from Lemma
\ref{lem:u-xi} and the equation satisfied by $u_j$. The time Lipschitz estimate
 of the $\xi_j$'s is obtained in the same way, using
the fact that we can bound the right hand side of the equation
satisfied by $\xi_j$. Indeed, one can use the space oscillation estimate of $u$
to bound $F(t,[u(t,\cdot)]_{j,m}(x))$ (as we did in \eqref{eq:boundF1}-\eqref
{eq:boundF2}) and Lemma \ref{lem:u-xi} and Proposition \ref{pro:130} to
bound remaining terms.
\end{proof}
\medskip

We finally have
\begin{pro}[\bf Sub- and super- Lipschitz continuous hull functions]\label{pro:139}
  We consi\-der $0<\delta\le 1$ and the Lipschitz continuous hull function
   obtained in Proposition \ref{pro:135} for
  $a_0=\pm 1$, that we call $((h^{\d,\pm}_{j} )_j,(g^{\d,\pm}_{j} )_j)$, and  the corresponding value $\lambda^{\d,\pm}$ of the
  effective Hamiltonian.  Then we have
\begin{eqnarray*}
(h^{\d,+}_{j})_\tau +\lambda^{\d,+} (h^{\d,+}_{j})_z &=& \alpha_0 (g_{j}^{\d,+} - h_{j}^{\d,+}) \\
(g^{\d,+}_{j})_\tau +\lambda^{\d,+} (g^{\d,+}_{j})_z &\ge&  2 F_j(\tau,
[h^{\d,+}(\tau,\cdot)]_{j,m}) + \alpha_0 (h_{j}^{\d,+} - g_{j}^{\d,+}) 
\end{eqnarray*}
and
$$
\lambda \le \lambda^{\d,+}  \to \lambda \quad \mbox{as}\quad  \delta
\to 0
$$
and
\begin{eqnarray*}
(h^{\d,-}_{j})_\tau +\lambda^{\d,-} (h^{\d,-}_{j})_z &=& \alpha_0 (g_{j}^{\d,-} - h_{j}^{\d,-}) \\
(g^{\d,-}_{j})_\tau +\lambda^{\d,-} (g^{\d, -}_{j})_z &\ge&  2 F_j(\tau,
[h^{\d,-}(\tau,\cdot)]_{j,m}) + \alpha_0 (h_{j}^{\d,-} - g_{j}^{\d,-}) 
\end{eqnarray*}
and
$$
\lambda \ge \lambda^{\d,-}  \to \lambda \quad \mbox{as}\quad  \delta
\to 0
$$
where $\lambda=\overline{F}(p)$.
\end{pro}

\begin{proof}[Proof of Proposition \ref{pro:139}]

\noindent
Inequalities $\pm \lambda^{\d,\pm} \ge \pm \lambda$ follow from the
comparison principle. Remark that bounds (\ref{eq:140}) and (\ref{eq:138}) on
$\lambda^{\d,\pm}$ and $h^{\d,\pm}_{j}$  are uniform as $\delta$ goes
to zero. Hence  the convergence $\lambda^{\d,\pm} \to \lambda$ holds
true as $\delta \to 0$. 
Indeed, it suffices to adapt Step~2 of the proof of Proposition \ref{pro:129}.
\end{proof}

\section{Qualitative properties of the effective Hamiltonian}\label{s6}
\setcounter{equation}{0}

\begin{proof}[Proof of Theorem \ref{th:4}]

\noindent
We recall that we have hull functions $((h_j)_j,(g_j)_j)$ solutions of
$$\left\{\begin{array}{l}
(h_j)_\tau +\lambda (h_j)_z= \a_0 (g_j-h_j)\\
(g_j)_\tau +\lambda (g_j)_z=2L+2F(\tau,[h(\tau,\cdot)]_{j,m}(z))+\a_0(h_j-g_j)
\end{array}\right.$$
with $\lambda=\overline{F}(L,p)$.

\noindent
The continuity of the map $(L,p)\mapsto \overline{F}(L,p)$ is easily proved
as in step 2 of the proof of Proposition \ref{pro:129}.

\noindent {\bf (i) Bound}

\noindent
This is a straightforward adaptation of the proof of \eqref{eq:boundlambda}.

\noindent {\bf (ii) Monotonicity in $L$}

\noindent
The monotonicity of the map $L\mapsto \overline{F}(L,p)$ follows from the
comparison principle on 
$$((u_j(\tau,y)=h_j(\tau,\lambda\tau+py))_j,(\xi_j(\tau,y)=g_j(\tau,\lambda\tau+py))_j $$ 
where $((h_j)_j,(g_j)_j)$ is the
hull function and $\lambda=\overline{F}(L,p)$.

\end{proof}

\appendix


\section{An alternative proof of Proposition~\ref{pro:130}}
\renewcommand{\theequation}{A.\arabic{equation}}
\setcounter{theo}{0}
\setcounter{equation}{0}
\renewcommand{\thesubsection}{A.\arabic{subsection}}
In this section, we give an alternative proof of Proposition~\ref{pro:130}. 
We adapt here the method we used in \cite{FIM2} and we provide complementary details. 

\subsection{Explanation of the estimate of  Proposition~\ref{pro:130}}\label{sec:appB}
In this section, we formally explain how we derive the estimate obtained in 
 Proposition~\ref{pro:130}.

\noindent
We can adapt the corresponding proof from \cite{FIM2}. 
For all $\eta\ge 0$, we consider the following Cauchy problem
\begin{equation}\label{eq:131}
\left\{\begin{array}{l}
\left\{\begin{array}{rl}
(u_j)_\tau &= \alpha_0 (\xi_j - u_j) \\
(\xi_j)_\tau&=G_j^\delta(\tau,
[u(\tau,\cdot)]_{j,m},\xi_j(\tau, y),\inf_{y'\in\R}\left(\xi_j(\tau,y')-py'\right)+py-\xi_j(\tau,y),(\xi_j)_y)
+ \eta (\xi_j)_{yy} \\
\end{array}\right. 
\\ \\
\left\{\begin{array}{rl}
u_{j+n}(\tau,y)&=u_j(\tau,y+1)\\
\xi_{j+n} (\tau,y) & = \xi_j (\tau,y+1) \\
\end{array}\right. 
\\ \\
\left\{\begin{array}{rl}
u_j(0,y)&=p\left(y+ \frac jn\right)  \\
\xi_j(0,y)&=p\left(y+ \frac jn\right)   \\
\end{array}\right. 
\end{array}\right. 
\end{equation}
where $G_j^\delta$ is given by
$$
G_j^\delta(\tau,V,r,a,q)=2F_j(\tau,V)+\alpha_0 (V_0 -r) +\delta(a_0+a)q
$$
(remark that this is not exactly the function given by \eqref{defi:gdelta}).
It is convenient to introduce the modified Hamiltonian 
$$
\tilde{F}_i (\tau,V_{-m}, \dots, V_m) = 2 F_i (\tau,V_{-m}, \dots, V_m) + \alpha_0 V_0
$$
so that 
$$
G_j^\delta(\tau,V_{-m}, \dots, V_m ,r,a,q) = \tilde{F}_j (\tau,V_{-m},\dots, V_m)- \alpha_0 r +
\delta (a_0 + a) q \, .
$$
Hence, the Lipschitz constant of $\tilde{F}_j(\tau,V)$ with respect to
$V$ is $\tilde{K}_1 = 2 L_F + \alpha_0$.

\noindent {\bf Case A: $\eta >0$ and $F_j \in C^1$}
For $\eta >0$, it is possible to show that
there exists a unique solution $((u_j)_j,(\xi_j)_j)$ of (\ref{eq:131}) in
$(C^{2+\alpha,1+\alpha})^{2n}$ for any $\alpha\in (0,1)$.  We will give the main idea of this existence result in the next subsection.

\noindent {\bf Step 1: bound from below on the gradient}

\noindent
Then, if we define $\zeta_j = (\xi_j)_y$ and $v_j=(u_j)_y$, we can derive the previous
equation in order to get the following one
\begin{equation}
\left\{\begin{array}{rl}\label{eq:132}
(v_j)_\tau = &\alpha_0 (\zeta_j -v_j) \\
(\zeta_j)_\tau - \eta (\zeta_j)_{yy}=  & 
(\tilde{F_j})'_V(\tau,[u(\tau,\cdot)]_{j,m}(y))\cdot [v(\tau,\cdot)]_{j,m}(y) 
- \alpha_0 \zeta_j  - \delta (\zeta_j-p)\zeta_j\\
&  + \delta \left( a_0 +
\inf_{y'\in\R}\left( \xi_j(\tau,y')-py' \right) +py-\xi_j(\tau,y) \right) (\zeta_j)_y  \\
\\
v_{j+n} (\tau,y) =&v_j (\tau,y+1) \\
\zeta_{j+n}(\tau,y)=&\zeta_j(\tau,y+1) \\
\\
v_j(0,y) =&\zeta_j(0,y)=p \, .
\end{array}\right. 
\end{equation}

 Let us now define
$$
\underline{m}_v(\tau)=\inf_{j\in \{1,\dots,n\}}\inf_{y\in\R} v_j (\tau,y) 
\quad \text{and} \quad \underline{m}_\zeta(\tau)=\inf_{j\in \{1,\dots,n\}}\inf_{y\in\R} \zeta_j (\tau,y) \, .
$$
Then we have in the viscosity sense:
$$\left\{\begin{array}{l}
(\underline{m}_v )_\tau \ge \alpha_0 (\underline{m}_\zeta - \underline{m}_v) \\
(\underline{m}_\zeta)_\tau \ge \tilde L_F\min(0,\underline{m}_v)-\a_0 \underline m_\zeta  - \delta
(\underline{m}_\zeta-p)\underline{m}_\zeta   \\
\underline{m}_v(0)=\underline m_\zeta(0)=p>0
\end{array}\right.$$
where we have used the monotonicity assumptions (A2) and (A3) to get the term
$\tilde L_F\min(0,\underline{m}_v)$ with $\tilde L_F=2L_F+\a_0$.
The fact that $(0,0)$ is a sub-solution of this monotone system of ODEs
implies that, for $j\in\{1,\dots,n\}$,
$$
 v_j\ge \underline{m}_v \ge 0\quad {\rm and}\quad \zeta_j\ge \underline{m}_\zeta \ge 0 \, .
$$
In particular, we see that $(u,\xi)$ is a solution of \eqref{eq:131} with 
$G_j^\delta$ given by \eqref{defi:gdelta}.

\medskip

\noindent {\bf Step 2: bound from above on the gradient}

\noindent Similarly we define
$$
\overline{m}_\zeta (\tau)=\sup_{j\in\{1,\dots,n\}}\sup_{y\in\R} \zeta_j(\tau,y) \quad
\text{and} \quad \overline{m}_v (\tau)=\sup_{j\in\{1,\dots,n\}}\sup_{y\in\R} v_j(\tau,y)\, .
$$
Then we have in the viscosity sense
$$
\left\{\begin{array}{l}
(\overline{m}_v)_\tau \le \alpha_0 (\overline{m}_\zeta - \overline{m}_v ) \\
    (\overline{m}_\zeta)_\tau \le (2L_F) \overline{m}_v  
+\alpha_0 (\overline{m}_v - \overline{m}_\zeta) - \delta (\overline{m}_\zeta-p)\overline{m}_\zeta   \\
\overline{m}_v (0)=    \overline{m}_\zeta(0)=p>0
\end{array}\right.
$$
where we have used Step 1 to ensure that $v_j \ge \underline{m}_v \ge 0$ for $j\in\{1,\dots,n\}$.
The constant function $(p+ (2L_F)\delta^{-1})$ (for both components) is a super-solution of the 
previous monotone system of ODEs. Hence, the proof is complete in Case A. 
\medskip

\noindent {\bf Case B: $\eta =0$ and $F$ general}

We can use an approximation argument as in \cite{FIM2}. 
This ends the proof of the proposition.
\subsection{Proof of the existence of a regular solution of \eqref{eq:131}}

We just give the main idea.

It can be useful to remark that $u_{j+l}$ can be rewritten as follows: for all $l \in \{-m,\dots,m\}$,
\begin{equation}\label{eq:xi2u}
u_{j+l} (\tau,y) = p \cdot (y+(j+l)/n) e^{-\alpha_0 \tau} 
+ \alpha_0 \int_0^\tau e^{\alpha_0 (s-\tau)} \xi_{j+l} (s,y) ds \, .
\end{equation}
We set  $v_j(\tau, y)=\xi_j(\tau, y)-py$. Then $(v_j)_j$ is a solution of
\begin{equation}\label{eq:v}\left\{\begin{array}{l}
(v_j)_t-\eta (v_j)_{yy}=\overline F_j(t, [v(\tau,\cdot)+p\cdot]_{j,m}(y))
+\delta \left(1+\inf_{y'}(v(\tau,y'))-v(\tau,y)\right)(v_y+p)\\
v_{j+n}(\tau,y)=v_j(\tau, y+1)+p\\
v_j(0,y)=p(\frac j n)
\end{array}\right.
\end{equation}
where $\overline F_j[\tau,[\xi(\tau,\cdot)]_{j,m}(y)] = 2 F_j(\tau,[u(\tau,\cdot)]_{j,m}(y))+ \a_0 u_j(\tau,y)-\xi_j(\tau,y)$ with $u$ given by \eqref{eq:xi2u} as a function of the time integral of $\xi$.
Since we attempt to get $\xi_j(\tau, y+\frac 1p)=\xi_j(\tau, y)+1$, we
will look for functions $v_j$ which are periodic of period $\frac
1p$. The basic idea is to use a fixed point argument. First, we
``regularize'' the right hand side of \eqref{eq:v} by considering for
some given $K>0$
$$
\F_{K,j}(\tau, v)=T_K^0(\overline F_j(\tau, [v(\tau,\cdot)+p\cdot]_{j,m}(y)))
+\delta \left(1+T_K^1(\inf_{y'}(v(\tau,y'))-v(\tau,y))\right)(T_K^3(v_y+p))
$$
where $T_K^i\in C^\infty_b$ are truncature functions. In particular,
$\F_{K,j}(\tau,\cdot)\in W^{1,\infty}$ uniformly in $\tau\in[0,+\infty)$ and so for all $q>1$, there exists a solution
$w=(w_j)_j=A(v)\in W^{2,1;q}([0,T]\times [0,\frac 1p))$ of
$$ (w_j)_t-\eta (w_j)_{yy}=\F_{K,j}(v)$$
Now, we want to show that the operator $A$ is a contraction. Let $v_1,
v_2\in W^{2,1;q}([0,T]\times [0,\frac 1p))$. Standard parabolic estimates show that
\begin{align*}
  &|A_j(v_1)-A_j(v_2)|_{ W^{2,1;q}([0,T]\times [0,\frac 1p))}\\
  \le &C|\F_{K,j}(\tau, v_1)-\F_{K,j}(\tau, v_2)|_{L^q([0,T]\times [0,\frac 1p))}\\
  \le &C \left(|v_2-v_1|_{L^q([0,T]\times [0,\frac 1p))}
    +|\inf(v_2)-v_2-(\inf(v_1)-v_1)|_{L^q([0,T]\times [0,\frac 1p))}+|(v_2-v_1)_y|_{L^q([0,T]\times [0,\frac 1p))}\right)\\
  \le & C T^\beta|v_2-v_1|_{ W^{2,1;q}([0,T]\times [0,\frac 1p))}
\end{align*}
for some $\beta>0$ (see \cite{LSU67,IJM}). \medskip

Sobolev embedding and parabolic regularity theory in Holder's spaces implies the existence for $T$ small enough of a solution $w_j\in C^{2+\a,
  \frac { 2+\a}2}$.\medskip

While we have smooth solutions below the truncature, we can apply the
arguments of Subsection \ref{sec:appB} and get estimates on the
gradient of the solution which ensures that the solution is indeed
below the truncature. Finally, a posteriori, the truncature can be
completely removed because of our estimate on the gradient of the
solution.  \medskip

\noindent {\bf Acknowledgements}\\
This work was partially supported by the ANR-funded project ``MICA'' (2006-2010).
The second author was also partially supported by the ANR-funded project ``Kam Faible''
(2008-2012). 


\def\cprime{$'$}

\end{document}